\documentclass[preprint,11pt,3p]{elsarticle}
\usepackage{amsmath,amsthm,amsfonts,amssymb}
\usepackage{xcolor}
\usepackage{algorithm}
\usepackage{multirow}
\usepackage{multicol}
\usepackage{booktabs}
\usepackage{algpseudocode}

\newcommand{\tub}{\mathbf{tube}}

\newcommand{\mc}[1]{\mathcal {#1}}
\newcommand{\m}{\!*\!{\!_M}}

\newcommand{\mbc}{\mathbb{C}}

\usepackage{mathrsfs}
\usepackage{sectsty}
\definecolor{astral}{RGB}{46,116,181}
\sectionfont{\color{astral}}
\newtheorem{theorem}{Theorem}[section]
\newtheorem{lemma}[theorem]{Lemma}
\newtheorem{corollary}[theorem]{Corollary}
\newtheorem{proposition}[theorem]{Proposition}

\newtheorem{definition}[theorem]{Definition}
\newtheorem{example}[theorem]{Example}

\usepackage{hyperref}
\usepackage{subfigure}
\usepackage{scalerel}
\usepackage{stackengine,wasysym}
\newcommand\reallywidetilde[1]{\ThisStyle{
  \setbox0=\hbox{$\SavedStyle#1$}%
  \stackengine{-.1\LMpt}{$\SavedStyle#1$}{
    \stretchto{\scaleto{\SavedStyle\mkern.2mu\AC}{.30\wd0}}{.4\ht0}%
  }{O}{c}{F}{T}{S}%
}}

\def\wtilde#1{
  \reallywidetilde{#1}}
\usepackage{tikz,xcolor}
\definecolor{lime}{HTML}{A6CE39}
\definecolor{lightblue}{rgb}{0.0, 0.0, 0.5}
\DeclareRobustCommand{\orcidicon}{%
	\begin{tikzpicture}
	\draw[lime, fill=lime] (0,0)
	circle [radius=0.16]
	node[white] {{\fontfamily{qag}\selectfont \tiny ID}};
	\draw[white, fill=white] (-0.0625,0.095)
	circle [radius=0.007];
	\end{tikzpicture}
	\hspace{-2mm}
}
\foreach \x in {A, ..., Z}{%
	\expandafter\xdef\csname orcid\x\endcsname{\noexpand\href{https://orcid.org/\csname orcidauthor\x\endcsname}{\noexpand\orcidicon}}
}

\begin{document}
\begin{frontmatter}
\title{
Two-step parameterized tensor-based iterative methods for solving  $\mc{A}\m\mc{X}\m\mc{B}=\mc{C}$
}
\author{Ratikanta Behera$^a$, Saroja Kumar Panda$^{\dagger}$$^b$, Jajati Keshari Sahoo$^{\dagger}$$^c$}

\address{
 $^{a}$Department of Computational and Data Sciences, \\
 IISc, Bangalore, India.\\
\textit{E-mail}: \texttt{ratikanta@iisc.ac.in}\\

\vspace{.3cm}
 $^{\dagger}$Department of Mathematics,\\
Birla Institute of Technology Pilani, K.K. Birla Goa Campus, Zuarinagar, Goa, India
\\\textit{E-mail\,$^b$}: \texttt{p20170436@goa.bits-pilani.ac.in}\\
\textit{E-mail\,$^c$}: \texttt{jksahoo@goa.bits-pilani.ac.in  }\\

}

\begin{abstract}
Iterative methods based on tensors have emerged as powerful tools for solving tensor equations, and have significantly advanced across multiple disciplines. In this study, we propose two-step tensor-based iterative methods to solve the tensor equations $\mc{A}\m\mc{X}\m\mc{B}=\mc{C}$ 
by incorporating preconditioning techniques and parametric optimization to enhance convergence properties. The theoretical results were complemented by comprehensive numerical experiments that demonstrated the computational efficiency of the proposed two-step parametrized iterative methods. The convergence criterion for parameter selection has been studied and a few numerical experiments have been conducted for optimal parameter selection. Effective algorithms were proposed to compute iterative methods based on two-step parameterized tensors, and the results are promising. In addition, we discuss the solution of the Sylvester equations and a regularized least-squares solution for image deblurring problems.  

\end{abstract}

\end{frontmatter}
\section{Introduction}
A tensor is a mathematical object that generalizes scalars, vectors, and matrices to arbitrary dimensions, that is, the $n$-dimensional array of elements and each element is specified by the $n$ coordinates \cite{Wensheng,Jidrazin, Ragnarsson, SahooPandaBehera25, Stanimirovi}. In particular, a scalar is a single element ($0$th-order tensor), and a vector (column vector or row vector) is a one-dimensional array of elements ($1$st-order tensor). A matrix is a $2$nd-order tensor. A $3$rd-order tensor can be visualized as a cube of elements, and each element is specified by three coordinates.  Operations on tensors are the fundamental challenge, as they do not possess a canonical multiplication framework. Specifically, tensor products ( Einstein product \cite{Einstein} and the n-mode product \cite{Bader,qi2017}, M-product \cite{Kernfeldlinear}, $c$-product \cite{Kernfeldlinear}, $t$-product \cite{kilmer13,kilmer11}, and Shao's general product \cite{Shao1} ) present distinct mathematical frameworks, each characterized by unique computational paradigms and the resultant properties that require meticulous consideration in their application in various computational domains in machine learning and data science \cite{miwakeichi}.

We use calligraphic notation (e.g., $\mc{A}, \mc{B}, \mc{C}$, etc.) for tensors.  The elements of a third-order tensor $\mc{A}$ are represented as per the following:
\begin{equation*}
(\mc{A})_{ijk}=\mc{A}(i,j,k)=a_{ijk}, ~~~ 1\leq i \leq m, ~~1\leq j \leq n, ~~1\leq k \leq p
\end{equation*}
The collection of all these tensors defined over the complex field is represented as $\mathbb{C}$, or for the real field, we use \(\mathbb{R}\), specifically denoted as \(\mathbb{C}^{m \times n \times p}\) (or \(\mathbb{R}^{m \times n \times p}\)). In this context, the \(k\)th frontal slice of a tensor \(\mathcal{A} \in \mathbb{C}^{m \times n \times p}\) is expressed as \(\mathcal{A}^{(k)} = \mathcal{A}(:, :, k)\). Furthermore, the tube fibers of the tensor $\mathcal{A}$ can be expressed as $\mathcal{A}(i, j, :)$, $\mathcal{A}(i, :, k)$, or $\mathcal{A}(:, j, k)$.  
Brazell et al.'s pioneering work \cite{Brazell} on tensor inverse computations using the Einstein product framework established foundational methods for solving multilinear systems \cite{AsishRJ, behera2023computing,Jidrazin}. The field undergone significant advancements with Kilmer and Martin's introduction of $t$-product-based tensor multiplication \cite{Braman10,kilmer13,kilmer11}. It is essential to highlight that the $t$-product mentioned in \cite{kilmer13} arises from using the unnormalized DFT matrix as the matrix $M$. On the other hand, the $c$-product in \cite{kilm21} is created by defining $M$ as $W^{-1}C(I + Z)$. In this expression, $C$ is the Discrete Cosine Transform (DCT) matrix of size $n$,  $W$ is a diagonal matrix that takes the first column of $C$ and $Z=\mathrm{diag}(\text{ones}(n-1,1),1)$. We next define the face-product of the tensors $\mc{A}\in \mbc^{m\times n\times p}$ and $\mc{B}\in \mbc^{n\times k\times p}$ is denoted by $\mc{A}\Delta\mc{B}$ and defined element-wise as 
\begin{equation*}
(\mc{A}\Delta\mc{B})(:,:,i)=\mc{A}(:,:,i)\mc{B}(:,:,i),~i=1,2,\ldots,p.
\end{equation*}
Now we define the $M$-product of two tensors as follows:

\begin{definition}\label{DeftProd}{\rm \cite{Kernfeldlinear}}
Let $M\in\mathbb{C}^{p \times p}$ be an invertible matrix. The $M$-product of two tensors $\mathcal{A} \in \mathbb{C}^{m \times n \times p}$ and $\mathcal{B} \in \mathbb{C}^{n \times k \times p}$, produces a new tensor, denoted by $\mathcal{A}_{*M}\mathcal{B} \in \mathbb{C}^{m \times k \times p}$. This operation combines the two tensors in a specific manner, and is defined as follows:
\begin{equation*}
\mathcal{A}_{*M}\mathcal{B}=((\mathcal{A}\times_3 M)\Delta (\mathcal{B}\times_3 M))\times_3 M^{-1}, 
\end{equation*}
where $\mathcal{A}\times_3 M$ is called the $3$-mode product of  ${\mathcal{A}} \in \mathbb{C}^{m \times n \times p}$ and $M\in\mathbb{C}^{p \times p}$, which is defined element-wise as follows
\begin{equation*}
(\mathcal{A}\times_3 M)(i_1,i_2,k)=\sum_{i_3=1}^p \mathcal{A}(i_1,i_2,i_3)M(k,i_3)\quad \text{where}\quad k\in\{1,\ldots,p\}.
\end{equation*}
\end{definition}

The evolution of iterative methods for solving matrix equations has seen significant advancements in recent years. Building upon classical Jacobi and Gauss-Seidel matrix splitting techniques, researchers have developed increasingly sophisticated approaches to solve the matrix equation $AXB=C$, where the coefficients $A \in \mathbb{C}^{m \times m}$, $B\in \mathbb{C}^{n \times n}$, and $C \in \mathbb{C}^{m \times n}$ are known matrices and $X$ is the unknown matrix to be determined. These linear systems arise naturally as mathematical representations of diverse phenomena in engineering \cite{BouJb08,DuRuan22,Regalia} and scientific domains \cite{DuRuan22,WangLi13,ZhouDuan08}. Liu et al. \cite{TianLiu17} advanced the field by developing substantial improvement of iteration methods based on induced splitting of matrices $A$ and $B$, incorporating a flexible parametric framework that demonstrated significant computational advantages over the traditional progressive iterative approximation (PIA) method \cite{LiuGuo18} in computer-aided geometric design. Further, the authors of \cite{Chen21} discussed a two-step AOR iteration method to solve this equation. Recursive blocked algorithms \cite{ChenMin20} represented a significant advancement in the solution of Sylvester matrix equations and their generalizations.

Next, we consider the tensor representation of a multilinear system of the form
\begin{equation}\label{maineq}
\mc{A}\m\mc{X}\m\mc{B}=\mc{C},
\end{equation}
where $\mc{A}\in \mathbb{C}^{m \times m\times p}$, $\mc{B} \in \mathbb{C}^{n \times n\times p}$ both are non singular. The tensor equation \ref{maineq} represents a sophisticated mathematical framework that extends classical matrix operations to higher-dimensional spaces, providing a powerful formalism for modeling complex multidimensional relationships. Its fundamental structure enables the representation of intricate interactions between multiple variables and parameters in tensor form, surpassing traditional matrix-based approaches in capturing higher-order correlations and dependencies.

This leads researchers to develop various numerical methods \cite{behera2023computation,mqdr,SahooPandaBehera25}. These approaches can be broadly categorized into iterative methods, progressively refining approximations through systematic updates, and direct methods, which attempt to compute exact solutions through finite algebraic operations. While direct methods offer precise solutions under specific conditions, iterative approaches often provide greater flexibility and computational efficiency for large-scale problems, particularly when exact solutions are not required or are computationally prohibitive. This complementary development of both methodological paradigms has significantly enhanced our ability to address increasingly complex mathematical modeling challenges across various scientific and engineering domains. The tensor formulation proves particularly valuable in scenarios where traditional matrix approaches become insufficient to capture the full complexity of multidimensional data relationships \cite{SahooPandaBehera25}. This advancement has opened new avenues for applications across diverse fields, from signal processing and data analytics to computer vision and scientific computing, where the ability to model and manipulate high-dimensional data structures is paramount.

The principal contributions of this research can be summarized as follows:

\begin{itemize}
\item We introduce novel two-step iterative methods based on tensor formulations to solve tensor equations of the form $\mc{A}\m\mc{X}\m\mc{B}=\mc{C}$ by incorporating preconditioning techniques and parametric optimization approaches to improve convergence characteristics.

\item We present a detailed investigation of the convergence criteria for parameter selection, including the theoretical bounds for convergence and numerical experiments, to determine the optimal parameter values.

\item In view of two-step parameterized tensor-based iterative methods, efficient $M$-product-based algorithms have been designed to calculate the solutions. A few numerical experiments validated the effectiveness of our proposed methods, and the comparative performance metrics provided quantitative evidence of their efficiency.

\item We extend our methodology to address two significant applications:  the solution of Sylvester equations and the regularized least-squares approach for image deblurring problems.
\end{itemize}

In this paper, the outline of the organization of the content is as follows.\\
Section \ref{Section2} reviews the notation and fundamental concepts of the tensor, nonnegative matrices, matrix splittings, regularization techniques, block tensors, and the essential properties of tensors under various product operations. The core contribution of our work is presented in Section \ref{Section3}, where we introduce a novel parameterized iterative method based on two steps to find the solution of $\mc{A}\m\mc{X}\m\mc{B}=\mc{C}$. Section \ref{sec4} demonstrates the practical utility of our approach through its application to Sylvester equations and image deblurring problems, illustrating the effectiveness of the method. Section \ref{Section5} brings everything together by highlighting our key findings, and suggests exciting paths for future research.

\section{Preliminaries}\label{Section2}
This section provides the mathematical foundation by reviewing several key concepts. We first recall a few definitions to familiarize the $M$-product.  Next, we briefly discuss regularization and nonnegative matrices with their fundamental properties in matrix splittings. The section then delves into the theory of block tensors, presenting their structure and characteristics. Finally, we analyze the basic properties of tensors under various product operations, providing a comprehensive framework for the following developments.
\subsection{Notation and Definitions}
To simplify our presentation, we use additional notation to make subsequent mathematical concepts and operations more concise and easy to understand. For $\mc{A}\in \mbc^{m\times n\times p}$ and $M\in \mbc^{p\times p}$, we define 
\[\hat{\mc{A}}:=\mc{A}\times_{3}M \in \mbc^{m\times n \times p}.\]  It's important to understand that this ``hat" notation is linked to the particular matrix $M$ that we have used for the transformation. From this point forward, we will assume that $M$ is a nonsingular (or invertible) matrix. It is easy to verify that 
\begin{equation}\label{MP}
\mc{A}=\mc{A}\times_3 M\times_3 M^{-1}=\tilde{\mc{A}}\times_3 M^{-1}.
\end{equation}

The tubal spectral norm, also known as the tubal norm, and the tubal  rank \cite{Kernfeldlinear,kilmer13, kilm21} of a tensor are defined as follows:  The tubal norm of $\mc{A} \in \mbc^{m\times n\times p}$ is denoted by $\|\mc{A}\|_M$ and is defined as $\|\mc{A}\|_M=\displaystyle{\max_{1\leq i\leq p}(\|\tilde{\mc{A}}(:,:,i)\|_2)}$, where $M\in\mbc^{p\times p}$ is an arbitrary matrix. Further, the tubal rank is defined as $\max_{1\leq i\leq p}r_i$, where $r_i=\mathrm{rank}(A(:,:,i)$). The tensor $\mc{B}$ is called the transpose conjugate of $\mc{A}$ under the face-product if $\mc{B}(:,:,i)=\left[\mc{A}(:,:,i)\right]^*$ for each $i$, $i=1,2,\ldots,p$ and denoted by $\mc{A}^*_{\Delta}$.  A tensor \(\mathcal{Y} \in \mathbb{C}^{n \times m \times p}\) is referred to as the Moore-Penrose (MP) inverse of the tensor \(\mathcal{A}\) (denoted as \(\mathcal{A}_{\Delta}^{\dagger}\)) when it satisfies the following conditions related to the face-product operation.
\[\mc{A}\Delta\mc{Y}\Delta \mc{A}=\mc{A},~\mc{Y}\Delta\mc{A}\Delta \mc{Y}=\mc{Y},~(\mc{A}\Delta\mc{Y})^*_{\Delta}=\mc{A}\Delta\mc{Y},\mbox{ and }(\mc{Y}\Delta\mc{A})^*_{\Delta}=\mc{Y}\Delta\mc{A}\]
Suppose  $M\in \mbc^{p\times p}$, $\mc{A}\in \mbc^{m\times n\times p}$ and $\mc{B}\in \mbc^{n\times k\times p}$. Using Definition \ref{DeftProd} and equation \eqref{MP} one can write the $M$-product of the tensors $\mc{A}$ and $\mc{B}$ as follows
\begin{equation*}
\mc{A}\m\mc{B}=\left(\tilde{\mc{A}}\Delta\tilde{\mc{B}}\right)\times_3 M^{-1}.
\end{equation*}

\begin{definition}{\rm \cite{Kernfeldlinear}}
Consider $M\in\mbc^{p\times p}$ and  $\mc{A}\in \mbc^{m\times m\times p}$. Then we call $\mc{A}$ is
\begin{enumerate}
   \item[(i)] an identity tensor (denoted by $\mc{I}_{mmp}$) if for each $i$, $\tilde{\mc{A}}(:,:,i)$ is an identity matrix of order $m$.
    \item[(ii)] invertible or nonsingular if there exists a tensor $\mc{Z}$ such that 
    $\mc{A}\m\mc{Z}=\mc{Z}\m\mc{A}=\mc{I}_{mmp}$.
\end{enumerate}
\end{definition}
\begin{definition} \cite{Kernfeldlinear} 
Let $M\in\mbc^{p\times p}$ and $\mc{A}\in \mbc^{m\times n\times p}$. Then 
\begin{enumerate}
    \item[(i)]
$\mc{A}$ is said to be a 
    zero or null tensor (denoted by $\mc{O}$) if all the elements of  $\tilde{\mc{A}}$ are zero. 
    \item[(ii)] $\mc{A}^*=\wtilde{\mc{A}^*}\times_3M^{-1}$, where 
 $\wtilde{\mc{A}^*}=A^*_{\Delta}$.
 \end{enumerate}
\end{definition}
In \cite{jin2023}, the authors have introduced the concepts of the MP inverse under the $M$-product, which is defined in the next. For  $M\in\mbc^{p\times p}$ and  $\mc{A}\in \mbc^{m\times n\times p}$,  a tensor 
$\mc{Y} \in \mathbb{C}^{n\times m \times p}$ satisfying 
{\small{
\begin{eqnarray*}
\mc{A}\m\mc{Y}\m\mc{A} = \mc{A}, ~~\mc{Y}\m\mc{A}\m\mc{Y} = \mc{Y} ~~(\mc{A}\m\mc{Y})^* = \mc{A}\m\mc{Y},~~(\mc{Y}\m\mc{A})^* = \mc{Z}\m\mc{A}.
\end{eqnarray*}
}}
is called the MP inverse of $\mc{A}$ and is 
 denoted by $\mc{A}^{\dagger}.$ On the other hand, if a tensor $\mc{Z}$ satisfying $\mc{A}\m\mc{Z}\m\mc{A} = \mc{A}$
 is called an inner inverse  of $\mc{A}$, which is represented by $\mc{A}^{(1)}$.

\begin{definition}
Consider $M\in\mbc^{p\times p}$ and  $\mc{A}\in \mbc^{m\times n\times p}$. Then $\mc{A}$ is called
\begin{enumerate}
    \item[(i)]   diagonally dominant (strictly diagonally dominant) \cite{Sahoo} if  $\tilde{\mc{A}}(:,:,i)$ is diagonally dominant (strictly diagonally dominant) for all $i$, $1=1,2,\ldots,p$.
     \item[(ii)]  hermitian positive definite (HPD)  if  $\tilde{\mc{A}}(:,:,i)$ is HPD for all $i$, $1=1,2,\ldots,p$.
       \item[(iii)] nonnegative (denoted by $\mc{A}\geq 0$) if   \[(\tilde{\mc{A}})_{ijk}\geq 0 \mbox{ for all }1\leq i\leq m,~1\leq j\leq n,~1\leq k\leq p.\]

\end{enumerate}
 \end{definition}

For $\mc{A}\in \mbc^{m\times m \times p}$ and $M \in \mbc^{p\times p}$, $\lambda$ is called an eigenvalue of $\mc{A}$ if $\mc{A}\m\mc{X}=\lambda\mc{X}$ for some $\mc{X}\neq \mc{O}$. Further, the spectral radius  \cite{SahooPandaBehera25} of $\mc{A}$ under the $M$-produced is defined as 
${\displaystyle \rho (\mc{A})=\max_{1\leq i\leq p}\{\rho (\tilde{\mc{A}}(:,:,i})\}$, where $\rho (\tilde{\mc{A}}(:,:,i))$ is the spectral radius of $\tilde{\mc{A}}(:,:,i))$.

\subsection{Nonnegative matrices, splittings and regularization}
A matrix $A$ is called nonnegative if $a_{ij}\geq 0$ for all $i$ and $j$ varying in its ranges. In short, it is denoted by $A\geq 0$. The splitting of a matrix $A$ is a decomposition of the form $A= F-G$. For nonnegative matrices along with splittings, extensive research has been carried  out to find the approximate solution linear systems of the form $Ax=b$. Some of these can be found in previous studies \cite{berman,varga,wozi,young}.  Let us recall a few useful results of nonnegative matrices which are essential to prove some of our results for tensors under the $M$-product structure. 
\begin{lemma}\cite{varga1,varga}\label{lem-reg}
    Let $A=F-G$ be a regular splitting ( regular: if $F^{-1}\geq 0$ and $G\geq 0$ ) of a nonsingular matrix $A$. Then $A^{-1}\geq 0$ if and only if $\rho(F^{-1}G)<1$.
\end{lemma}
\begin{lemma}\cite{perea1,perea}\label{lem-weak}
    Let $A=F-G$ be a weak regular splitting ( weak regular: if $F^{-1}\geq 0$ and $F^{-1}G\geq 0$ ) of a nonsingular matrix $A$. Then $A^{-1}\geq 0$ if and only if $\rho(F^{-1}G)<1$.
\end{lemma}
\begin{lemma}\cite[Lemma 2.1]{song}\label{lem-non}
    Let $A$ be a nonsingular matrix and $A=F-G$ be nonnegative splitting (nonnegative: if $F^{-1}G\geq 0$ ). Then $A^{-1}F\geq 0$ if and only if $\rho(F^{-1}G)<1$.
\end{lemma}
\begin{lemma}\cite[Theorem 2.20]{varga}\label{preron}
Let $A\in\mathbb{R}^{m\times m}$ and $A\geq 0$. Then
\begin{itemize}
    \item[(i)] $A$ has a nonnegative real eigenvalue $\lambda$ such that  $\lambda=\rho(A)$.
    \item[(ii)] there exists $x(\geq 0)$ such that $Ax=\rho(A)x$.
\end{itemize}
\end{lemma}
Barata and Hussein \cite{barata2012moore} established the following essential result regarding matrices.
\begin{lemma}\cite{barata2012moore}\label{tik-mat}
Consider $A\in\mathbb{C}^{m\times n}$. Then \[\lim_{_{\lambda\to 0}}A^*(AA^*+\lambda\mc{I})^{-1}=\lambda\lim_{\lambda\to 0}(A^*A+\lambda I)A^*=A^{\dagger}.\]
\end{lemma}
The well-posed or ill-conditioned of a multilinear system defined in \cite{SahooPandaBehera25} as per the following: We call a tensor $\mc{A}\in\mathbb{C}^{m\times n\times p}$ is ill-conditioned if $\tilde{\mc{A}}(:,:,i)$  is ill-conditioned  for some $i$, where $ 1\leq i\leq k$. On the other hand, the multilinear system represented by \(\mc{A}\m\mc{X}=\mc{B}\) is termed well-posed (ill-posed) if $\wtilde{\mc{A}\m\mc{X}}(:,:,i)=\tilde{\mc{B}}(:,:,i)$ is well-posed (ill-posed) for all $i,~i=1,2\ldots, p$.  In  the framework of $M$-product, Sahoo et al. \cite{SahooPandaBehera25} have discussed Tikhonov's regularized solution  $\mc{X}_{\lambda}=(\mc{A}^*\m\mc{A}+\lambda \mc{I})^{-1} \mc{A}^*\m\mc{B}$ for the tensor equations of the form $\mc{A}\m\mc{X}=\mc{B}$. The same idea is extended for the tensor equations of the form $\mc{A}\m\mc{X}\m\mc{B}=\mc{C}$, in Section \ref{sec4}.
\subsection{Block tensor}
For  $\mc{B}\in \mbc^{m\times n \times p}$  and $\mc{C}\in \mbc^{m\times k \times p}$, the row block tensor $\begin{bmatrix}
    \mc{B}& \mc{C}
\end{bmatrix}\in\mbc^{m\times (n+k) \times p}$ of $\mc{B}$ and $\mc{C}$, is  defined element-wise as
\[\begin{bmatrix}
    \mc{B}& \mc{C}
\end{bmatrix}=\mc{A}\in\mbc^{m\times (n+k) \times p}, \mbox{ where }\mc{A}(:,:,i)=\begin{bmatrix}
    \mc{B}(:,:,i)& \mc{C}(:,:,i)
\end{bmatrix},~i=1,2,\ldots,p.\]
Similarly, the column block tensor $\begin{bmatrix}
    \mc{B}_1\\
    \mc{C}_1
\end{bmatrix}\in\mbc^{(m+k)\times n \times p}$ of $\mc{B}_1\in \mbc^{m\times n \times p}$ and $\mc{C}_1\in \mbc^{k\times n \times p}$, is defined as
\[\begin{bmatrix}
    \mc{B}_1\\
     \mc{C}_1
\end{bmatrix}=\begin{bmatrix}
    \mc{B}_1^T& \mc{C}_1^T
\end{bmatrix}^T\]
Utilizing both row  and column block tensors, we define the $2\times 2$ matrix  block tensor of $\mc{B}\in \mbc^{m\times n \times p}$, $\mc{C}\in \mbc^{m\times k \times p}$, $\mc{D}\in \mbc^{s\times n \times p}$,  and $\mc{E}\in \mbc^{s\times k\times p}$ as follows:
\[\begin{bmatrix}
    \mc{B}& \mc{C}\\
     \mc{D}& \mc{E}\\
\end{bmatrix}=\mc{A}\in \mbc^{(m+s)\times (n+k) \times p}, \mbox{ where }\mc{A}(:,:,i)=\begin{bmatrix}
    \mc{B}(:,:,i)& \mc{C}(:,:,i)\\
      \mc{D}(:,:,i)& \mc{E}(:,:,i)\\
\end{bmatrix},~i=1,2,\ldots,p.\]
From the definition of block tensors, we can show the following results.
\begin{proposition}\label{pbls}
For $1\leq i\leq 6$, let $\mc{A}_i,~\mc{B}_i,~\mc{C}_i,~\mc{D}_i$ be tensors having consistent with multiplication. Then:
\begin{enumerate}
\item[(i)] $\mc{A}_1\m\begin{bmatrix}
    \mc{B}_1&\mc{C}_1
\end{bmatrix}=\begin{bmatrix}
    \mc{A}_1\m\mc{B}_1&\mc{A}_1\m\mc{C}_1
\end{bmatrix}$.
\item[(ii)]  $\begin{bmatrix}
    \mc{A}_2\\
    \mc{B}_2
\end{bmatrix}\m\mc{C}_2=\begin{bmatrix}
    \mc{A}_2\m\mc{C}_2\\\mc{B}_2\m\mc{C}_2
\end{bmatrix}$.
\item[(iii)] $\begin{bmatrix}
    \mc{A}_3&\mc{B}_3
\end{bmatrix}\m\begin{bmatrix}
    \mc{C}_3\\
    \mc{D}_3
\end{bmatrix}= \mc{A}_3\m\mc{C}_3+\mc{B}_3\m\mc{D}_3$.
\item[(iv)] $\begin{bmatrix}
    \mc{A}_4\\
    \mc{B}_4
\end{bmatrix}\m\begin{bmatrix}
    \mc{C}_3&\mc{D}_4
\end{bmatrix}= \begin{bmatrix}
    \mc{A}_4\m\mc{C}_4&\mc{A}_4\m\mc{D}_4\\
    \mc{B}_4\m\mc{C}_4&\mc{B}_4\m\mc{D}_4
\end{bmatrix}$.
\item[(v)] $\begin{bmatrix}
    \mc{A}_5 & \mc{B}_5\\
    \mc{C}_5&\mc{D}_5
\end{bmatrix}\m\begin{bmatrix}
    \mc{D}_1\\
    \mc{D}_2
\end{bmatrix}= \begin{bmatrix}
    \mc{A}_5\m\mc{D}_1+\mc{B}_5\m\mc{D}_2\\
    \mc{C}_5\m\mc{D}_1+\mc{D}_5\m\mc{D}_2
\end{bmatrix}$.
\item[(vi)] $\begin{bmatrix}
    \mc{D}_1&
    \mc{D}_2
\end{bmatrix}\m\begin{bmatrix}
    \mc{A}_6 & \mc{B}_6\\
    \mc{C}_6&\mc{D}_6
\end{bmatrix}= \begin{bmatrix}
    \mc{D}_1\m\mc{A}_6+\mc{D}_2\m\mc{C}_6&
   \mc{D}_1\m\mc{B}_6+\mc{D}_2\m\mc{D}_6
\end{bmatrix}$.
\end{enumerate}
\end{proposition}
\begin{proposition}\label{pblsp}
    Let $\mc{B}\in \mbc^{m\times n \times p}$, $\mc{D}\in \mbc^{m\times k \times p}$, $\mc{O}\in \mbc^{s\times n \times p}$,  and $\mc{C}\in \mbc^{s\times k\times p}$. If $\mc{A}=\begin{bmatrix}
    \mc{B}& \mc{D}\\
     \mc{O}& \mc{C}\\
\end{bmatrix}$ then $\rho(\mc{A})=\max\{\rho(\mc{B}),\rho(\mc{C})\}$.
\end{proposition}
\subsection{Basic Properties of tensors under different products}
Using the definition of face-product, we can show the following.
\begin{proposition}\label{prop-face}
    Let $\mc{A}\in \mbc^{m\times n\times p}$ and $\mc{B}\in \mbc^{n\times k\times p}$. Then 
    \begin{enumerate}
        \item[(i)] $(\mc{A}\Delta\mc{B})^*_{\Delta}=\mc{B}^*_{\Delta}\Delta \mc{A}^*_{\Delta}$.
        \item[(ii)] If $\mc{X}=\mc{A}_{\Delta}^{\dagger}$ then $\mc{X}(:,:,i)=\left[\mc{A}(:,:,i)\right]^{\dagger}$ for $i=1,2,\ldots,p$.
    \end{enumerate}
\end{proposition}
\begin{proposition}\cite[Proposition 1.3]{mqdr}\label{prop-basic}
Let ${\mc{C}} \in \mbc^{m \times n \times p}$, ${\mc{D}} \in \mbc^{n \times q \times p}$ and $M\in\mbc^{p \times p}$. Then 
    \begin{enumerate}
        \item[(i)] $\tilde{\mc{C}}\Delta\tilde{\mc{D}}=(\mc{C}\m\mc{D})\times_3M$.
        \item[(ii)] $(\tilde{\mc{C}})^*=\wtilde{\mc{C}^*}=\mc{C}^*\times_3M$.
        \item[(iii)] $\wtilde{\mc{C}+\mc{D}}=\tilde{\mc{C}}+\tilde{\mc{D}}$.
        \item[(iv)] $\left(\wtilde{\mc{C}\m\mc{D}}\right)^{-1}=(\tilde{\mc{D}})^{-1}\Delta(\tilde{\mc{C}})^{-1}$.
        \item[(v)] $\left({\mc{C}\m\mc{D}}\right)^{-1}=\mc{D}^{-1}\m 
\mc{C}^{-1}$.
     \end{enumerate}
\end{proposition}

\begin{proposition}\label{prop-tildag}
Let ${\mc{A}} \in \mbc^{m \times n \times p}$, ${\mc{B}} \in \mbc^{n \times k \times p}$ and $M\in\mbc^{p \times p}$. Then 
\begin{enumerate}
    \item[(i)] $(\mc{A}\m\mc{B})^*=\mc{B}^*\m\mc{A}^*$.
    \item[(ii)] $(\tilde{\mc{A}})_{\Delta}^{\dagger}=\wtilde{\mc{A}^{\dagger}}$.
    \item[(iii)] $(\tilde{\mc{A}})^{\dagger}=\wtilde{\mc{A}^{\dagger}}$ if and only if $\mc{A}^{\dagger}=\mc{A}_{\Delta}^{\dagger}$.
\end{enumerate}
\end{proposition}
\begin{proof}
(i) Applying Propositions \ref{prop-face} and  \ref{prop-basic}, we have  
\[(\mc{B}^*\m\mc{A}^*)\times_3 M=\tilde{\mc{B}^*}\Delta\tilde{\mc{A}^*}=(\tilde{\mc{B}})^*\Delta(\tilde{\mc{A}})^*= \left(\tilde{\mc{A}}\Delta\tilde{\mc{B}}\right)^*=\left(\wtilde{\mc{A}\m\mc{B}}\right)^*=\left(\mc{A}\m\mc{B}\right)^*\times_3 M.\]
Post multiplying by $M^{-1}$ (with respect to $3$-mode product), we obtain $(\mc{A}\m\mc{B})^*=\mc{B}^*\m\mc{A}^*$.\\
(ii) Let $\mc{Z}=\wtilde{\mc{A}^{\dagger}}$. Then 
\[\tilde{\mc{A}}=\wtilde{\mc{A}\m\mc{A}^{\dagger}\m\mc{A}}=\tilde{\mc{A}}\Delta\wtilde{\mc{A}^\dagger}\Delta\tilde{\mc{A}}=\tilde{\mc{A}}\Delta\mc{Z}\Delta\tilde{\mc{A}},\]
\[\mc{Z}\Delta\tilde{\mc{A}}\Delta\mc{Z}=\wtilde{\mc{A}^{\dagger}}\Delta\tilde{\mc{A}}\wtilde{\mc{A}^{\dagger}}=\wtilde{\mc{A}^{\dagger}\m\mc{A}\m\mc{A}^{\dagger}}=\wtilde{\mc{A}^{\dagger}}=\mc{Z},\]
\[\left(\tilde{\mc{A}}\Delta\mc{Z}\right)^*=\left(\tilde{\mc{A}}\Delta\wtilde{\mc{A}^{\dagger}}\right)^*=\left(\wtilde{\mc{A}\m\mc{A}^{\dagger}}\right)^*=\wtilde{(\mc{A}\m\mc{A}^{\dagger})^*}=\wtilde{\mc{A}\m\mc{A}^{\dagger}}=\tilde{\mc{A}}\Delta\wtilde{\mc{A}^{\dagger}}=\tilde{\mc{A}}\Delta\mc{Z},\]
and 
\[\left(\mc{Z}\Delta\tilde{\mc{A}}\right)^*=\left(\wtilde{\mc{A}^{\dagger}}\Delta\tilde{\mc{A}}\right)^*=\left(\wtilde{\mc{A}^{\dagger}\m\mc{A}}\right)^*=\wtilde{(\mc{A}^{\dagger}\m\mc{A})^*}=\wtilde{\mc{A}^{\dagger}\m\mc{A}}=\tilde{\wtilde{\mc{A}^{\dagger}\Delta\mc{A}}}=\mc{Z}\Delta\tilde{\mc{A}}.\]
Thus $(\tilde{\mc{A}})_{\Delta}^{\dagger}=\mc{Z}=\wtilde{\mc{A}^{\dagger}}$.\\
(iii) Let $\mc{A}^{\dagger}=\mc{A}_{\Delta}^{\dagger}$ and $\mc{Z}=\wtilde{\mc{A}^{\dagger}}$. Then 
\begin{eqnarray*}
 \left(\tilde{\mc{A}}\m\mc{Z} \m\tilde{\mc{A}} \right)\times_3 M^{-1}&=&  \left(\tilde{\mc{A}}\m\wtilde{\mc{A}^{\dagger}} \m\tilde{\mc{A}} \right)\times_3 M^{-1}=(\tilde{\mc{A}}\times_3 M^{-1})\Delta(\wtilde{\mc{A}^{\dagger}}\times_3 M^{-1})\Delta(\tilde{\mc{A}}\times_3 M^{-1})\\
&=&\mc{A}\Delta\mc{A}^{\dagger}\Delta\mc{A}=\mc{A}=\tilde{\mc{A}}\times_3 M^{-1}.
\end{eqnarray*}
Post multiplying by $M^{-1}$ (with respect to $3$-mode product), we obtain $\tilde{\mc{A}}\m\mc{Z} \m\tilde{\mc{A}}=\tilde{\mc{A}}$. Similarly, $\mc{Z}\m\tilde{\mc{A}} \m\mc{Z}=\mc{Z}$ is follows from 
\begin{eqnarray*}
    (\mc{Z}\m\tilde{\mc{A}} \m\mc{Z})\times_3 M^{-1}=\mc{A}^{\dagger}\Delta\mc{A}\Delta\mc{A}^{\dagger}=\mc{A}^{\dagger}=\wtilde{\mc{A}^{\dagger}}\times_3 M^{-1}=\mc{Z}\times_3 M^{-1}.
\end{eqnarray*}
Now 
\begin{eqnarray*}
 \left(\tilde{\mc{A}}\m\mc{Z}\right)^*\times_3 M^{-1}&=&\left(\mc{Z}^*\m\wtilde{(\mc{A}^*)}\right)\times_3 M^{-1}=\wtilde{(\mc{A}^{\dagger})^*}\times_3M^{-1}\Delta\wtilde{(\mc{A})^*}\times_3M^{-1} =(\mc{A}^{\dagger})^*\Delta \mc{A}^*\\
 &=&\left(\mc{A}\Delta\mc{A}^{\dagger}\right)^*=\left(\tilde{\mc{A}}\m\wtilde{\mc{A}^{\dagger}}\right)\times_3M^{-1}= \left(\tilde{\mc{A}}\m\mc{Z}\right)\times_3 M^{-1}.
\end{eqnarray*}
Post multiplying by $M^{-1}$ (with respect to $3$-mode product), we obtain $\left(\tilde{\mc{A}}\m\mc{Z}\right)^*=\tilde{\mc{A}}\m\mc{Z}$. Similarly, we can show $\left(\tilde{\mc{Z}}\m\mc{A}\right)^*=\tilde{\mc{Z}}\m\mc{A}$. Hence  $(\tilde{\mc{A}})^{\dagger}=\mc{Z}=\wtilde{\mc{A}^{\dagger}}$.\\ 
The converse part can be verified similarly. 
\end{proof}
\begin{lemma}\cite{SahooPandaBehera25}\label{conlemma}
       Consider $M \in \mbc^{p\times p}$ and  $\mc{A}\in \mbc^{m\times m \times p}$. If $\|\mc{A}\|_M<1$ then $\lim_{k\to \infty}\mc{A}^k=\mc{O}$. In addition  $\rho(\mc{A})<1\iff \lim_{k\to \infty}\mc{A}^k=\mc{O}$, where $\mc{O}\in\mbc^{m\times m \times p}$ is the zero tensor.
\end{lemma}

\section{Two-step parameterized iterative method} \label{Section3}
Now, we present a novel tensor-based two-step parameterized iterative method for solving multilinear systems. This innovative approach advances the field by providing a systematic framework that leverages the structural properties of tensors while incorporating parameterization to enhance the computational efficiency and accuracy of the solution. Specifically, we develop iterative methods for the tensor equation \eqref{maineq}. 
First, we transform the tensor equation \eqref{maineq} into the following two coupled tensor equations.
\[
\mc{A}\m \mc{Y}=\mc{C} \text{~and~}
\mc{X}\m \mc{B}=\mc{Y} .
\] 
We split the tensors $\mc{A}$ and $\mc{B}$ using nonsingular tensors as follows.
\[\mc{A}=\mc{F}_1-\mc{G}_1 \text{~and~}  \mc{B}=\mc{F}_2-\mc{G}_2,  \]
where $\mc{F}_1$ and $\mc{F}_2$ are nonsingular tensors. Introducing two positive parameters $\alpha$ and $\beta$,  the parameterized tensor splittings of $\mc{A}$ and $\mc{B}$, respectively written as 
\[\mc{A}=\frac{\mc{F}_1}{\alpha}-\bar{\mc{G}}_1 \text{~and~} 
\mc{B}=\frac{\mc{F}_2}{\beta}-\bar{\mc{G}}_2 .    \]
Using the splitting  $\frac{\mc{F}_1}{\alpha}-\bar{\mc{G}}_1$ in solving  $\mc{A}\m \mc{Y}=\mc{C}$, we obtain the following iterative scheme
\begin{align}\label{iteq1}
\nonumber
\mc{Y}_{k+1} & =\alpha \mc{F}_1^{-1} \m \bar{\mc{G}}_1 \m\mc{Y}_k+\alpha \mc{F}_1^{-1}\m \mc{C}=\alpha \mc{F}_1^{-1}\m\left(\frac{\mc{F}_1}{\alpha}-\mc{A}\right)\m \mc{Y}_k+\alpha\mc{F}_1^{-1}\m\mc{C}\\
&=\left(\mc{I}_{mmp}-\alpha\mc{F}_1^{-1}\m\mc{A} \right)\m\mc{Y}_k+\alpha\mc{F}_1^{-1}\m\mc{C}.
\end{align}
Similarly, by using the splitting $\frac{\mc{F}_2}{\beta}-\bar{\mc{G}}_2$ for  $\mc{X}\m \mc{B}=\mc{Y}$, we obtain the following iterative scheme
\begin{equation}\label{iteq2}
\mc{X}_{k+1}=\mc{X}_k\m\left(\mc{I}_{nnp}-\beta\mc{B}\m\mc{F}_2^{-1} \right)+ \beta\mc{Y}\m\mc{F}_2^{-1} .
\end{equation}
In view of \eqref{iteq1} and \eqref{iteq2}, we consider the following two-step parametrized iterative method for solving  tensor equations, of the form $ \mc{A}\m\mc{X}\m\mc{B}=\mc{C}$.
\begin{equation}\label{tpit}
\left\{\begin{array}{ll}\mc{Y}_{k+1}&=\left(\mc{I}_{mmp}-\alpha\mc{F}_1^{-1}\m\mc{A} \right)\m\mc{Y}_k+\alpha\mc{F}_1^{-1}\m\mc{C}\\
\mc{X}_{k+1}&=\mc{X}_k\m\left(\mc{I}_{nnp}-\beta\mc{B}\m\mc{F}_2^{-1} \right)+ \beta\mc{Y}_{k+1}\m\mc{F}_2^{-1} 
\end{array}
\right.    
\end{equation}
Next, we discuss the convergence analysis of the two-step parametrized iterative scheme \eqref{tpit}.
\begin{lemma}\label{lem1}
Consider $M\in\mbc^{p\times p}$, $\mc{A}$ and $\mc{C}$ as defined in \eqref{maineq} and $\left\{\mc {Y}_k\right\}$ as defined in \eqref{iteq1}. Then $\left\{\mc {Y}_k\right\}$ converges to $\mc{A}^{-1}\m\mc{C}$  for any initial approximation $\mc{Y}_0$, 
 if $\rho\left(\mc{I}_{mmp}-\alpha \mc{F}_1^{-1} \m\mc{A}\right)<1$.
\end{lemma}
\begin{proof}    
 Let us define $\tub:\mbc^{m\times n\times p}\mapsto\mbc^{mn\times 1\times p}$ by  $\mc{A}\mapsto \tub(\mc{A})$, where 
 \[\tub(\mc{A})(:,1,i)=\begin{bmatrix}
     \mc{A}(1,:,i),&\mc{A}(2,:,i),\cdots,\mc{A}(m,:,i)
 \end{bmatrix}^T\]
 Similarly, we can define $\tub^{-1}$ in such way that $\tub^{-1}(\tub(\mc{A}))=\mc{A}=\tub(\tub^{-1}(\mc{A}))$.
 Denoting $\mc{Z}_{k}=\tub(\mc{Y}_{k})$, $\mc{T}_1=\mc{I}_{mmp}-\alpha \mc{F}_1^{-1} \m\mc{A}$, and $\mc{C}_1=\tub(\alpha \mc{F}_1^{-1}\m\mc{C})$, the proposed iterative scheme \eqref{iteq1} becomes
 \[\mc{Z}_{k+1}=\mc{T}\m\mc{Z}_{k}+\mc{C}_1,\]
 where $\mc{T}(:,:,i)=I_{n}\otimes \mc{T}_1(:,:,i)$ for $i=1,2,\ldots,p$.  Since $\rho\left(\mc{T}(:,:,i)\right)=\rho\left(I_{m}\otimes \mc{T}_1(:,:,i)\right)=\rho\left(\mc{T}_1(:,:,i)\right)$, so we have 
\[\rho(\mc{T})=\max_{1\leq i\leq p}\{\rho (\mc{T}(:,:,i)\}=\max_{1\leq i\leq p}\{\rho (\mc{T}_1(:,:,i))\}=\rho(\mc{T}_1)=\rho\left(\mc{I}_{mmp}-\alpha \mc{F}_1^{-1} \m\mc{A}\right)<1\]
Hence, according to Theorem 6.10  
 \cite{SahooPandaBehera25}, $\mc{Z}_{k}=\tub(\mc{Y}_k)$ converges to $\tub(\mc{A}^{-1}\m\mc{C})$. Applying the inverse operation $\tub^{-1}$, we find that the sequence $\mc{Y}_k$ converges to the exact solution $\mc{A}^{-1}\m\mc{C}$.\\
{\bf Alternative Proof:} \\
 Let $\Psi_{ k+1} =\mc{Y} _{ k+1} - \mc{A}^{-1}\m\mc{C}$. Then 
 \begin{align}\label{eqpsi}
\nonumber
\Psi_{ k+1} &=\mc{Y} _{ k+1} - \mc{A}^{-1}\m\mc{C}=\mc{Y}_k-\mc{A}^{-1}\m\mc{C}+\alpha\mc{F}_{1}^{-1}\m(\mc{C}-\mc{A}\m\mc{Y}_k)\\
\nonumber
&=\Psi_k+\alpha\mc{F}_{1}^{-1}\m(\mc{A}\m\mc{A}^{-1}\m\mc{C}-\mc{A}\m\mc{Y} _{ k})=\Phi_k-\alpha\mc{F}_{1}^{-1}\m\mc{A}\m\Phi_k\\
&=(\mc{I}_{mmp}-\alpha\mc{F}_{1}^{-1}\m\mc{A})\m\Psi_k.
\end{align}
Applying equation \eqref{eqpsi} recursively, we obtain 
\begin{center}
    $\Psi_{ k+1}=(\mc{I}_{mmp}-\alpha\mc{F}_{1}^{-1}\m\mc{A})^{k+1}\m\Psi_0$. 
\end{center}
Since $\rho\left(\mc{I}_{mmp}-\alpha \mc{F}_1^{-1} \m\mc{A}\right)<1$, so by Lemma \ref{conlemma}, we get $\Psi_k\to 0$ as $k\to \infty$ and completes the proof.
\end{proof}
In case of the iterative scheme \eqref{iteq2}, we can similarly obtain the following result.
\begin{lemma}\label{lem2}
Consider $M\in\mbc^{p\times p}$, $\mc{B}$ as defined in \eqref{maineq} and $\left\{\mc {X}_k\right\}$ be defined as in \eqref{iteq2}. Then $\left\{\mc {X}_k\right\}$ converges to  $\mc{Y}\m\mc{B}^{-1}$  
for any initial approximation $\mc{X}_0$, if $\rho\left(\mc{I}_{nnp}-\beta \mc{B}\m \mc{F}_2^{-1}\right)<1$.
\end{lemma}
Using Lemma \ref{lem1} and Lemma \ref{lem2}, we discuss  the convergence of the proposed two-step parametrized iterative scheme \eqref{tpit}, in the next result.
\begin{theorem}\label{thtpit}
Consider $M\in\mbc^{p\times p}$, $\mc{A}$, $\mc{B}$,  $\mc{C}$ as defined in \eqref{maineq}. Let $\left\{\mc {X}_k\right\}$ and   $\left\{\mc {Y}_k\right\}$  as defined in \eqref{tpit}. Then $\left\{\mc{X}_k\right\}$  converges  to $\mc{A}^{-1}\m\mc{C}\m\mc{B}^{-1}$  if $\rho\left(\mc{I}_{mmp}-\alpha \mc{F}_1^{-1} \m\mc{A}\right)<1$ and  $\rho\left(\mc{I}_{nnp}-\beta \mc{B} \m \mc{F}_2^{-1}\right)<1$, for any initial approximation $\mc{X}_0$ with $\mc{X}_0\m\mc{B}=\mc{Y}_0$.
\end{theorem}
\begin{proof}
   Let $\Phi_k= \mc{X}_k-\mc{A}^{-1}\m\mc{C}\m\mc{B}^{-1}$ and $\Psi_k=\mc{Y}_k-\mc{A}^{-1}\m\mc{C}$. From \eqref{tpit}, we obtain 
   \begin{align}\label{eqphi}
   \nonumber
   \Phi_{ k+1}&=\mc{X}_{k+1}-\mc{A}^{-1}\m\mc{C}\m\mc{B}^{-1}= \mc{X}_k- \mc{A}^{-1}\m\mc{C}\m\mc{B}^{-1}+\beta(\mc{Y}_{k+1}-\mc{X}_k \m\mc{B})\m\mc{F}_{2}^{-1}\\
   \nonumber
   &= \Phi_{ k}-\beta(\mc{X}_k -\mc{A}^{-1}\m\mc{C}\m\mc{B}^{-1})\m\mc{B}\m\mc{F}_{2}^{-1}+\beta(\mc{Y}_{k+1}-\mc{A}^{-1}\m\mc{C})\m\mc{F}_{2}^{-1}\\
   &=\Phi_{ k}\m(\mc{I}_{nnp}-\beta\mc{B}\m\mc{F}_{2}^{-1})+\beta \Psi_{k+1}\m\mc{F}_{2}^{-1}.
 \end{align}
Applying \eqref{eqpsi} to the equation \eqref{eqphi}, and combining both we obtain the following system of tensor equations:
\begin{equation}\label{eqfisi}
\left\{\begin{array}{ll} 
\Phi_{ k+1}&=\Phi_{ k}\m(\mc{I}_{nnp}-\beta\mc{B}\m\mc{F}_{2}^{-1})+\beta (\mc{I}_{mmp}-\alpha\mc{F}_{1}^{-1}\m\mc{A})\m\Psi_k\m\mc{F}_{2}^{-1}\\
\Psi_{ k+1} &=(\mc{I}_{mmp}-\alpha\mc{F}_{1}^{-1}\m\mc{A})\m\Psi_k
\end{array}
\right.    
\end{equation}
Let $\mc{T}_1=\left(\mc{I}_{nnp}-\beta\mc{B}\m\mc{F}_{2}^{-1}\right)^T$, $\mc{T}_2=\mc{I}_{mmp}-\alpha\mc{F}^{-1}\m\mc{A}$ and $\mc{T}_3=\beta \left(\mc{F}_2^{-1}\right)^T$. Define $\mc{B}$, $\mc{C}$, and $\mc{D}$ with respective entries are given by 
\[\mc{B}(:,:,i)=\mc{T}_1(:,:,i)\otimes I_{m},~\mc{C}(:,:,i)=I_{n}\otimes\mc{T}_2(:,:,i),~\mc{D}(:,:,i)= \mc{T}_3(:,:,i)\otimes\mc{T}_2(:,:,i),~i=1,2,\ldots, p.\]
One can notice that 
\[\rho(\mc{B})=\max_{1\leq i\leq p}\{\rho (\mc{B}(:,:,i)\}=\max_{1\leq i\leq p}\{\rho (\mc{T}_1(:,:,i))\}=\rho(\mc{T}_1)=\rho\left(\mc{I}_{mmp}-\alpha \mc{F}_1^{-1} \m\mc{A}\right)<1,\]
and
\[\rho(\mc{C})=\max_{1\leq i\leq p}\{\rho (\mc{T}(:,:,i)\}=\max_{1\leq i\leq p}\{\rho (\mc{T}_2(:,:,i))\}=\rho(\mc{T}_2)=\rho\left(\mc{I}_{nnp}-\beta\mc{B}\m\mc{F}_{2}^{-1}\right)<1.\]
Let $\phi_{k+1}=\tub(\Phi_{k+1})$ and $\psi_{k+1}=\tub(\Psi_{k+1})$  Then equation \eqref{eqfisi} reduces to the following block tensor form
\[
  \begin{bmatrix}
    \phi_{k+1}\\
    \psi_{k+1}
\end{bmatrix}=\begin{bmatrix}
    \mc{B}& \mc{D}\\
    \mc{O}&\mc{C}
\end{bmatrix}\m\begin{bmatrix}
    \phi_{k}\\
    \psi_{k}
\end{bmatrix}=\mc{T}\m\begin{bmatrix}
    \phi_{k}\\
    \psi_{k}
\end{bmatrix}=\cdots=\mc{T}^{k+1}\m\begin{bmatrix}
    \phi_{0}\\
    \psi_{0}
\end{bmatrix}, 
\]
where $\mc{T}=\begin{bmatrix}
    \mc{B}& \mc{D}\\
    \mc{O}&\mc{C}
\end{bmatrix}$. Applying Proposition \ref{pblsp}, we have  $\rho(\mc{T})<1$ and consequently by Lemma \ref{conlemma}, we obtain $\tub(\Phi_{k+1})=\phi_{k+1}\to \mc{O}$ as $k\to\infty$. From inverse operation $\tub^{-1}$, it follows that $\Phi_{k+1}\to \mc{O}$ as $k\to\infty$ and completes the proof.
\\
{\bf Alternative Proof:} \\
Equation \eqref{eqphi} can be simplified as 
\begin{eqnarray*}
 \Phi_{ k+1}&=&\Phi_{ 0}\m\left(\mc{I}_{nnp}-\beta\mc{B}\m\mc{F}_{2}^{-1}\right)^{k+1}\\
 &&\hspace{1cm}+\beta\sum_{i=0}^{k+1}\left(\mc{I}_{mmp}-\alpha\mc{F}^{-1}\m\mc{A}\right)^{k+1-i}\Phi_0\m\mc{F}_2^{-1}\m\left(\mc{I}_{nnp}-\beta\mc{B}\m\mc{F}_{2}^{-1}\right)^{i}.   
\end{eqnarray*}
Taking $\mc{S}=\mc{I}_{nnp}-\beta\mc{B}\m\mc{F}_{2}^{-1}$ and $\mc{T}=\mc{I}_{mmp}-\alpha\mc{F}^{-1}\m\mc{A}$ and applying tubal norm, we obtain
\[\|\Phi_{ k+1}\|_M\leq \|\Phi_{ 0}\|_M\|\mc{S}\|_M^{k+1}+\beta \|\Psi_{ 0}\|_M\|\mc{F}_2^{-1}\|_M\sum_{i=0}^{k+1}\|\mc{S}\|_M^{k+1-i}\|\mc{T}\|_M^{i}.\]
Since $\|\mc{S}\|_M\leq \rho(\mc{S})<1$ and $\|\mc{T}\|_M\leq \rho(\mc{T})<1$, so  $\displaystyle \sum_{k=0}^{\infty}\|\mc{S}\|_M^{k}$ and $\displaystyle\sum_{k=0}^{\infty}\|\mc{T}\|_M^{k}$ both are convergent. Thus Cauchy product $\left(\displaystyle \sum_{k=0}^{\infty}\|\mc{S}\|_M^{k}\right)\left(\displaystyle\sum_{k=0}^{\infty}\|\mc{T}\|_M^{k} \right)=\displaystyle\sum_{k=0}^{\infty}c_k$ ( where $c_k=\displaystyle\sum_{i=0}^{k}\|\mc{S}\|_M^{k-i}\|\mc{T}\|_M^{i}$ ) is convergent. Hence $\sum_{i=0}^{k}\|\mc{S}\|_M^{k-i}\|\mc{T}\|_M^{i}=c_k\to 0$ as $k\to\infty$. By applying Lemma \ref{conlemma}, we conclude that  $\|\Phi_{ k+1}\|_M\to 0$ as $k\to\infty$.
\end{proof}
In view of Theorem \ref{thtpit}, the following algorithm is developed to compute the solution of \eqref{maineq}.
\begin{algorithm}[H]
 \caption{Two-step parameterized iterative method for solving $\mc{A}\m\mc{X}\m\mc{B}=\mc{C}$} \label{alg:PAIMethod}
\begin{algorithmic}[1]
\State {\bf Input} $\mc{A}\in \mathbb{C}^{m \times m\times p}$, $\mc{B} \in \mathbb{C}^{n \times n\times p}$,  $\mc{C} \in \mathbb{C}^{m \times n\times p}$,   $M\in\mbc^{p\times p}$.
\State {\bf Initial guess }$\mc{X}_0 \in \mathbb{C}^{m \times n\times p}$ and $\mc{Y}_0=\mc{X}_0\m\mc{B}$
\State {\bf Compute } Splittings of $\mc{A}=\mc{F}_{1}-\mc{G}_{1}$ and $\mc{B}=\mc{F}_{2}-\mc{G}_{2}$
\If{$\rho\left(\mc{I}_{mmp}-\alpha \mc{F}_1^{-1} \m\mc{A}\right)<1$ and  $\rho\left(\mc{I}_{nnp}-\beta \mc{B} \m \mc{F}_2^{-1}\right)<1$}
\While{({true})}
\State $\mc{Y}_{k+1}=\left(\mc{I}_{mmp}-\alpha\mc{F}_{1}^{-1} \m \mc{A} \right)\m\mc{Y}_k+\alpha\mc{F}_{1}^{-1}\m\mc{C}$
\State $ \mc{X}_{k+1}=\mc{X}_k\m\left(\mc{I}_{nnp}-\beta\mc{B}\m\mc{F}_{2}^{-1} \right)+ \beta\mc{Y}_{k+1}\m\mc{F}_{2}^{-1}$. 
   \If{tol $\leq \epsilon$}
   \State\textbf{break} 
  \EndIf
  \State $\mc{Y}_0\gets \mc{Y}_k$ and $\mc{X}_0\gets \mc{X}_k$
   \EndWhile 
   \Else~~ "The scheme does not converge"
   \EndIf
\State \Return $\mc{X}_k$ 
 \end{algorithmic}
\end{algorithm}
\begin{theorem}\label{thm-cgt}
Consider $M\in\mbc^{p\times p}$, $\mc{A}$, $\mc{B}$,  $\mc{C}$ as defined in \eqref{maineq}. Let $\left\{\mc {X}_k\right\}$ and   $\left\{\mc {Y}_k\right\}$  be defined as in \eqref{tpit}.  Let $0<\alpha<\frac{2}{1+\rho(\mc{F}_{1}^{-1}\m\mc{G}_1)}$ and $0<\beta<\frac{2}{1+\rho(\mc{F}_{2}^{-1}\m\mc{G}_2)}$.  If $\rho(\mc{F}_{1}^{-1}\m\mc{G}_1)<1$ and $\rho(\mc{F}_{2}^{-1}\m\mc{G}_2)<1$ then   $\left\{\mc{X}_k\right\}$  converges  to the exact solution $\mc{A}^{-1}\m\mc{C}\m\mc{B}^{-1}$ for any initial tensor $\mc{X}_0$ with $\mc{Y}_0=\mc{X}_0\m\mc{B}$.
\end{theorem}
\begin{proof}
From the splitting of $\mc{A}=\mc{F}_1-\mc{G}_1$, we get
\[\mc{I}_{mmp}-\alpha\mc{F}_{1}^{-1}\m\mc{A}=\mc{I}_{mmp}-\alpha\mc{F}_{1}^{-1}\m(\mc{F}_1-\mc{G}_1)=(1-\alpha)\mc{I}_{mmp}+\alpha\mc{F}_{1}^{-1}\m\mc{G}_1.\]
If  $\lambda_i (1\leq i\leq mp)$ is an  eigenvalue of the tensor $\mc{F}_{1}^{-1}\m \mc{G}_1$, then 
 \begin{equation}\label{eq8}
  \rho (\mc{I}_{mmp}-\alpha\mc{F}_{1}^{-1}\m\mc{A})= \max_{1 \leq i \leq mp} |1-\alpha+\alpha\lambda_i|\leq |1-\alpha|+\alpha\rho(\mc{F}_{1}^{-1}\m \mc{G}_1)
 \end{equation}
For $0 < \alpha \leq 1$ and applying equation \eqref{eq8}, we have 
 \begin{equation}\label{eq9}
 \rho (\mc{I}_{mmp}-\alpha\mc{F}_{1}^{-1}\m\mc{A})\leq 1-\alpha + \alpha\rho(\mc{F}_{1}^{-1}\m \mc{G}_1)< 1.
 \end{equation}
Similarly, for $1<\alpha< \frac{2}{1+\rho(\mc{F}_{1}^{-1}\m \mc{G}_1)}$, and combining with equation \eqref{eq8}, we get 
 \begin{equation}\label{eq10}  
 \rho (\mc{I}-\alpha\mc{F}_{1}^{-1}\m\mc{A})\leq  \alpha-1+\alpha\rho(\mc{F}_{1}^{-1}\m \mc{G}_1)< \frac{2}{1+\rho(\mc{F}_{1}^{-1}\m \mc{G}_1)}-1+ \frac{2\rho(\mc{F}_{1}^{-1}\m \mc{G}_1)}{1+\rho(\mc{F}_{1}^{-1}\m \mc{G}_1)}=1.
 \end{equation}
 Thus $\rho(\mc{I}_{mmp}-\alpha\mc{F}_{1}^{-1}\m\mc{A})< 1$  follows from  equation \eqref{eq9} and \eqref{eq10}. In a similar manner, we can show that $
  \rho (\mc{I}_{nnp}- \beta\mc{B}\m\mc{F}_{2}^{-1})< 1$. Hence, according to Theorem \ref{thtpit}, $\left\{\mc{X}_k\right\}$  converges  to the exact solution $\mc{A}^{-1}\m\mc{C}\m\mc{B}^{-1}$.
\end{proof}
Next, we discuss the convergence of the two-step parameterized iterative method for suitable choices of $\mc{F}_i$ and $\mc{G}_i$, $i=1,2$.  For $A\in\mbc^{m\times m\times p}$ and $A\in\mbc^{n\times n\times p}$, we first define the following tensors called the diagonal, strictly lower triangular and strictly upper triangular, respectively,  as follows:
\[\mc{D}_1(:,:,i)=diag(\mc{A}(:,:,i)), ~\mc{L}_1(:,:,i)=lowerdiag(\mc{A}(:,:,i)),~\mc{U}_1=\mc{A}-\mc{L}_1,~i=1,2,\ldots,p.\]
\[\mc{D}_2(:,:,i)=diag(\mc{B}(:,:,i)), ~\mc{L}_2(:,:,i)=lowerdiag(\mc{B}(:,:,i)),~\mc{U}_2=\mc{B}-\mc{L}_2,~i=1,2,\ldots,p.\]
Thus we can write the tensors $\mc{A}$ and $\mc{B}$ as $\mc{A}=\mathscr{F}_1-\mathscr{G}_1$ and $\mc{B}=\mathscr{F}_2-\mathscr{G}_2$, where 
\begin{equation}\label{eq11}
    \left\{\begin{array}{ll} 
\mathscr{F}_1&=\frac{1}{\omega_1}\left(\mc{D}_1+\kappa_1\mc{L}_1\right), ~~\mathscr{G}_1=\frac{1}{\omega_1}\left((1-\omega_1)\mc{D}_1+(\kappa_1-\omega_1)\mc{L}_1-\omega_1 \mc{U}_1\right)\\
\mathscr{F}_2&=\frac{1}{\omega_2}\left(\mc{D}_2+\kappa_2\mc{L}_2\right), ~~\mathscr{G}_2=\frac{1}{\omega_2}\left((1-\omega_2)\mc{D}_2+(\kappa_2-\omega_2)\mc{L}_2-\omega_2 \mc{U}_2\right)
\end{array}
\right.    
\end{equation}
Applying equation \eqref{tpit}  for the specific tensor splittings defined in equation \eqref{eq11}, we obtain the following iterative method (called the accelerated overrelaxation two-step parameterized iterative (AOR-TSPI) method).
\begin{equation}\label{eq12}
\left\{\begin{array}{ll}\mc{Y}_{k+1}&=\left(\mc{I}_{mmp}-\alpha\left(\frac{1}{\omega_1}\mc{D}_1+\frac{\kappa_1}{\omega_1}\mc{L}_1\right)^{-1}\m\mc{A} \right)\m\mc{Y}_k+\alpha\left(\frac{1}{\omega_1}\mc{D}_1+\frac{\kappa_1}{\omega_1}\mc{L}_1\right)^{-1}\m\mc{C}\\
\mc{X}_{k+1}&=\mc{X}_k\m\left(\mc{I}_{nnp}-\beta\mc{B}\m\left(\frac{1}{\omega_2}\mc{D}_2+\frac{\kappa_2}{\omega_2}\mc{L}_2\right)^{-1} \right)+ \beta\mc{Y}_{k+1}\m\left(\frac{1}{\omega_2}\mc{D}_2+\frac{\kappa_2}{\omega_2}\mc{L}_2\right)^{-1} 
\end{array}
\right.    
\end{equation}
Based on the AOR-TSPI method, the following algorithm is established. 
\begin{algorithm}[H]
 \caption{ AOR-TSPI method for solving $\mc{A}\m\mc{X}\m\mc{B}=\mc{C}$} \label{alg:AOR-TSPI}
\begin{algorithmic}[1]
\State {\bf Input} $\mc{A}\in \mathbb{C}^{m \times m\times p}$, $\mc{B} \in \mathbb{C}^{n \times n\times p}$,  $\mc{C} \in \mathbb{C}^{m \times n\times p}$,   $M\in\mbc^{p\times p}$ $\alpha,~\beta,~\omega_1,~\omega_2,~\kappa_1,~\kappa_2$.
\State {\bf Initial guess }$\mc{X}_0 \in \mathbb{C}^{m \times n\times p}$ and $\mc{Y}_0=\mc{X}_0\m\mc{B}$
\State Compute $\mc{D}_1,~\mc{D}_2,~\mc{L}_1,~\mc{L}_2$
\While{({true})}
\State $\mc{Y}_{k+1}=\left(\mc{I}_{mmp}-\alpha\left(\frac{1}{\omega_1}\mc{D}_1+\frac{\kappa_1}{\omega_1}\mc{L}_1\right)^{-1}\m\mc{A} \right)\m\mc{Y}_k+\alpha\left(\frac{1}{\omega_1}\mc{D}_1+\frac{\kappa_1}{\omega_1}\mc{L}_1\right)^{-1}\m\mc{C}$
  \State $ \mc{X}_{k+1}=\mc{X}_k\m\left(\mc{I}_{nnp}-\beta\mc{B}\m\left(\frac{1}{\omega_2}\mc{D}_2+\frac{\kappa_2}{\omega_2}\mc{L}_2\right)^{-1} \right)+ \beta\mc{Y}_{k+1}\m\left(\frac{1}{\omega_2}\mc{D}_2+\frac{\kappa_2}{\omega_2}\mc{L}_2\right)^{-1}  
  $
   \If{tol $\leq \epsilon$}
   \State\textbf{break} 
  \EndIf
  \State $\mc{Y}_0\gets \mc{Y}_k$ and $\mc{X}_0\gets \mc{X}_k$
   \EndWhile 
\State \Return $\mc{X}_k$ 
 \end{algorithmic}
\end{algorithm}
The AOR-TSPI method is  further classified for different choices of $\alpha$, $\beta$, $\omega_i$ and $\kappa_i$ as per the following Table \ref{tab:aortspi}.
\begin{table}[H]
    \centering
    \caption{Classification of AOR-TSPI methods}
    \vspace{0.2cm}
    \begin{tabular}{|c|c|}
    \hline
       Method Name  &  Parameters\\
       \hline
         Higher order Jacobi two-step iterative&\\
         (HOJ-TSI) method  & $\alpha=\beta=\omega_1=\omega_2=1,~\kappa_1=\kappa_2=0$\\
       \hline
         Higher order Jacobi two-step parametrized iterative&\\
         (HOJ-TSPI) method  & $\omega_1=\omega_2=1,~\kappa_1=\kappa_2=0$\\
         \hline
         Higher order Gauss-Seidel two-step  iterative&\\
         (HOGS-TSI) method  & $\alpha=\beta=\omega_1=\omega_2=\kappa_1=1=\kappa_2$\\
          \hline
         Higher order Gauss-Seidel two-step parametrized iterative&\\
         (HOGS-TSPI) method  & $\omega_1=\omega_2=1,~\kappa_1=1=\kappa_2$\\
         \hline
         Higher order successive overrelaxation  two-step parametrized&\\iterative
         (HOSOR-TSPI) method  & $\omega_1=\kappa_1,~\omega_2=\kappa_2$\\
         \hline
    \end{tabular}
    \label{tab:aortspi}
\end{table}
The convergence of AOR-TSPI methods are discussed in the next results, which are special cases of the Theorem \ref{thm-cgt}.
\begin{theorem}
Consider $M\in\mbc^{p\times p}$, $\mc{A}$, $\mc{B}$,  $\mc{C}$ as defined in \eqref{maineq}. Let $0<\alpha<\frac{2}{1+\rho\left(\mathscr{F}_1^{-1} \m\mathscr{G}_1\right)}, ~0<\beta<\frac{2}{1+\rho\left(\mathscr{F}_2^{-1}\m \mathscr{G}_2\right)}$.  If $\mc{A}$ and $\mc{B}$ are strictly diagonally dominant tensors then both HOJ-TSPI and HOGS-TSPI methods are convergent.
\end{theorem}
\begin{example}\rm\label{ex-3.6}
We randomly consider tensors $\mc{A},~\mc{B},~\mc{C}\in\mbc^{n\times n\times\times n}$, where tensors $\mc{A}$ and $\mc{B}$ both have a dominant diagonal property. The matrix $M$ is generated using the Matlab function $M=rand(n)$. A comparison analysis of the mean CPU time (MT) for different sizes of $\mc{A}$, $\mc{B}$, and parameters is presented in Table \ref{tab-hosgj} and Table \ref{tab1-hosgj}.
\end{example}
\begin{table}[H]
    \centering
    \caption{Comparison of error and mean CPU time for HOJ-TSPI and HOGS-TSPI for different tensors\\ and parameters where $\epsilon =10^{-15}$ }
    \vspace{0.2cm}
    \renewcommand{\arraystretch}{1.2}
    \begin{tabular}{c|c|c|c|c|c|c}
    \hline
 $\alpha=\beta$&   Method& $\mc{A}\in\mbc^{m\times m\times p}$& & & &\\
 ($\alpha$)& &
  $\mc{B}\in\mbc^{n\times n\times p}$
 & IT& \scriptsize{$\|\mc{C}-\mc{A}\m\mc{X}_{k}\m\mc{B}\|$} & \scriptsize{$\|\mc{X}_{k}-\mc{A}^{-1}\m\mc{C}\m\mc{B}^{-1}\|$}& MT\\
 \hline 
 $0.9$&HOJ-TSPI&$m=n=p=100$ &9.58&$3.9395e^{-12}$&$2.8840e^{-16}$&77\\
 $0.9$&HOGS-TSPI&$m=n=p=100$ &3.22&$5.5994e^{-13}$&$3.4152e^{-16}$&23  \\
 \hline
 $0.9$&HOJ-TSPI&$m=n=p=150$ &42.14&$8.6087e^{-12}$&$2.8063e^{-16}$&79\\
 $0.9$&HOGS-TSPI&$m=n=p=150$ &13.28&$2.2713e^{-12}$&$1.8466e^{-16}$&23  \\
 \hline
 $0.9$&HOJ-TSPI&$m=n=p=200$ &134.57&$1.2909e^{-12}$&$2.3627e^{-16}$&80\\
 $0.9$&HOGS-TSPI&$m=n=p=200$ &44.82&$9.7739e^{-13}$&$3.2426e^{-17}$&24  \\
 \hline
  $0.75$&HOJ-TSPI&$m=n=p=100$ &4.73&$ 1.3412e^{-12}$&$9.7249e^{-17}$&37\\
 $0.75$&HOGS-TSPI&$m=n=p=100$ &4.51&$ 1.1910e^{-13}$&$1.6670e^{-16}$&31  \\
 \hline
 $0.75$&HOJ-TSPI&$m=n=p=150$ &19.71&$ 4.5162e^{-12}$&$1.4615e^{-16}$&37\\
 $0.75$&HOGS-TSPI&$m=n=p=150$ &19.37&$3.4486e^{-12}$&$2.1169e^{-16}$&31  \\
 \hline
 $0.75$&HOJ-TSPI&$m=n=p=200$ &64.48&$1.1080e^{-11}$&$2.0256e^{-16}$&37\\
 $0.75$&HOGS-TSPI&$m=n=p=200$ &59.49&$1.9951e^{-13}$&$7.3316e^{-17}$&32 \\
 \hline
   $0.5$&HOJ-TSPI&$m=n=p=100$ &7.59&$ 3.0273e^{-12}$&$7.4155e^{-16}$&60\\
 $0.5$&HOGS-TSPI&$m=n=p=100$ &8.42&$2.7457e^{-13}$&$5.9922e^{-16}$&59  \\
 \hline
 $0.5$&HOJ-TSPI&$m=n=p=150$ &30.39&$ 9.7139e^{-12}$&$1.0418e^{-15}$&59\\
 $0.5$&HOGS-TSPI&$m=n=p=150$ &33.05&$8.2652e^{-12}$&$7.6546e^{-16}$&59  \\
 \hline
 $0.5$&HOJ-TSPI&$m=n=p=200$ &100.03&$1.4849e^{-11}$&$ 8.8465e^{-16}$&59\\
 $0.5$&HOGS-TSPI&$m=n=p=200$ &107.37&$2.145e^{-12}$&$3.1534e^{-16}$&59\\
 \hline
    \end{tabular}
    \label{tab-hosgj}
\end{table}
\begin{table}[H]
    \centering
    \caption{A comparison of the errors and average CPU timee for both HOJ-TSPI and HOGS-TSPI across various tensors and parameters, with taking the tolerance  $\epsilon =10^{-10}$ }
    \vspace{0.2cm}
    \renewcommand{\arraystretch}{1.2}
    \begin{tabular}{c|c|c|c|c|c|c}
    \hline
 $\alpha=\beta$&   Method& $\mc{A}\in\mbc^{m\times m\times p}$& & & &\\
 ($\alpha$)& &
  $\mc{B}\in\mbc^{n\times n\times p}$
 & IT& \scriptsize{$\|\mc{C}-\mc{A}\m\mc{X}_{k}\m\mc{B}\|$} & \scriptsize{$\|\mc{X}_{k}-\mc{A}^{-1}\m\mc{C}\m\mc{B}^{-1}\|$}& MT\\
 \hline 
 $0.95$&HOJ-TSPI&$m=n=p=100$ &9.57&$3.9958e^{-7}$&$2.9211e^{-11}$&77\\
 $0.95$&HOGS-TSPI&$m=n=p=100$ &2.46&$1.3020e^{-8}$&$2.5792e^{-12}$&17  \\
 \hline
 $0.95$&HOJ-TSPI&$m=n=p=200$ &132.45&$1.6024e^{-6}$&$2.9335e^{-11}$&80\\
 $0.95$&HOGS-TSPI&$m=n=p=200$ &30.32&$1.1134e^{-7}$&$5.9120e^{-12}$&17 \\
 \hline
 $0.95$&HOJ-TSPI&$m=n=p=300$ &614.99&$4.0467e^{-6}$&$3.2944e^{-11}$&81\\
 $0.95$&HOGS-TSPI&$m=n=p=300$  &154.78&$3.8832e^{-7}$&$9.0461e^{-12}$&17 \\
 \hline
  $0.85$&HOJ-TSPI&$m=n=p=100$ &5.46&$4.1306e^{-7}$&$3.0140e^{-11}$&43\\
 $0.85$&HOGS-TSPI&$m=n=p=100$ &2.58&$6.6664e^{-8}$&$8.8243e^{-12}$&18\\
 \hline
 $0.85$&HOJ-TSPI&$m=n=p=200$ & 74.79&$ 1.3266e^{-6}$&$2.4257e^{-11}$&45\\
 $0.85$&HOGS-TSPI&$m=n=p=200$ &32.48&$4.8565ee^{-7}$&$1.5349e^{-11}$&18 \\
 \hline
 $0.85$&HOJ-TSPI&$m=n=p=300$ &360.80&$ 2.3240e^{-6}$&$1.8920e^{-11}$&46\\
 $0.85$&HOGS-TSPI&$m=n=p=300$ &188.70&$ 1.0246e^{-6}$&$ 2.3325e^{-12}$&19 \\
 \hline
  $0.75$&HOJ-TSPI&$m=n=p=100$ &3.35&$ 2.9794e^{-7}$&$ 2.1749e^{-11}$&26\\
 $0.75$&HOGS-TSPI&$m=n=p=100$ &3.24&$1.6153e^{-7}$&$3.1980e^{-11}$&22  \\
 \hline
 $0.75$&HOJ-TSPI&$m=n=p=200$ &47.17&$ 4.4262e^{-7}$&$8.0766e^{-12}$&28\\
 $0.75$&HOGS-TSPI&$m=n=p=200$ &41.17&$ 3.2999e^{-7}$&$1.4047e^{-11}$&23 \\
 \hline
 $0.75$&HOJ-TSPI&$m=n=p=300$ &216.78&$ 1.2725e^{-6}$&$1.0348e^{-11}$&28\\
 $0.75$&HOGS-TSPI&$m=n=p=300$ &196.33&$ 1.0246e^{-6}$&$1.9386e^{-11}$&23 \\
 \hline
    \end{tabular}
    \label{tab1-hosgj}
\end{table}

\begin{example}\rm\label{ex-3.6.1}
We randomly consider tensors $\mc{A},~\mc{B}\in\mbc^{n\times n\times\times n}$,  where tensors $\mc{A}$ and $\mc{B}$ both have a dominant diagonal property. The matrix $M$ is generated using the Matlab function $M=rand(n)$ and the tensor $\mc{C}$ is chosen with entries that are $c_{ijk}=1$ for $1\leq i,j,k\leq n$. A comparison analysis of the mean CPU time (MT) for different sizes of $\mc{A}$, $\mc{B}$, and $\mc{C}$ for specific values of the parameters $\alpha$ and $\beta$ are presented in Figure \ref{Fig2d_comp}. In addition, we discussed the mean CPU time for a specific randomly generated tensors by varying the parameters in Figures \ref{Fig2d_alpha_comp}.
\end{example}
\begin{figure}[H]
\centering
\subfigure[$\alpha=0.8=\beta$]{\includegraphics[scale=0.45]{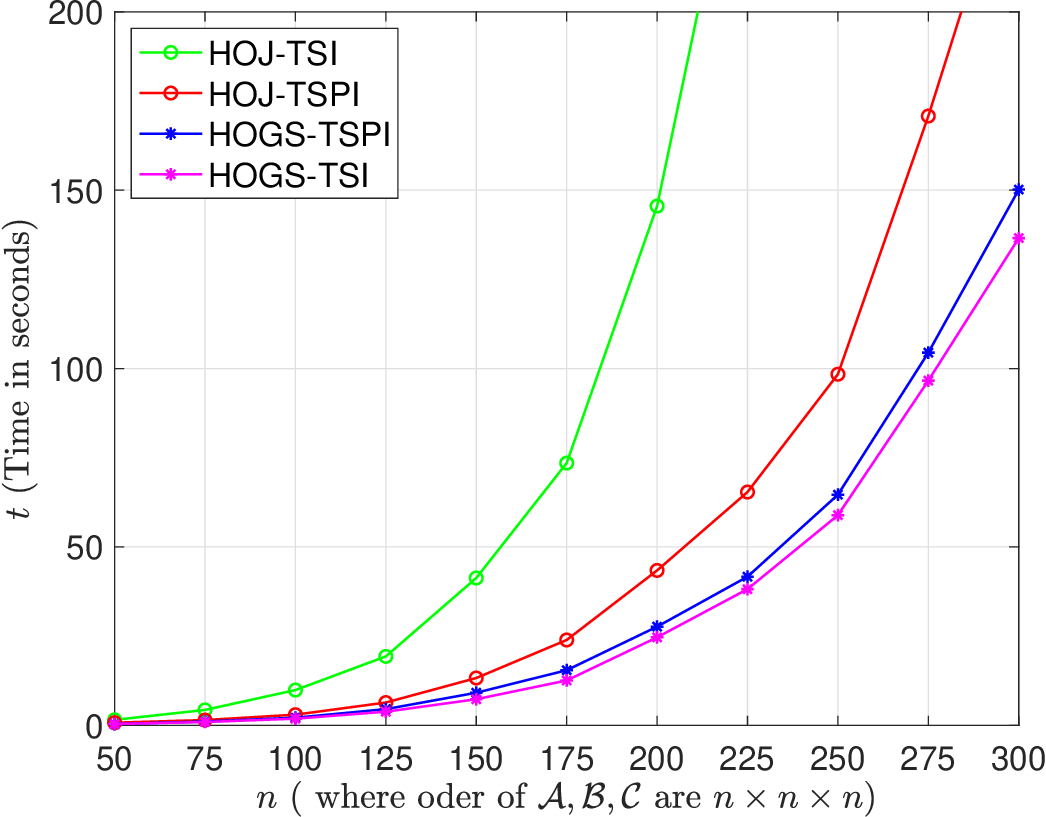}}
~~\subfigure[$\alpha=0.9=\beta$]{\includegraphics[scale=0.45]{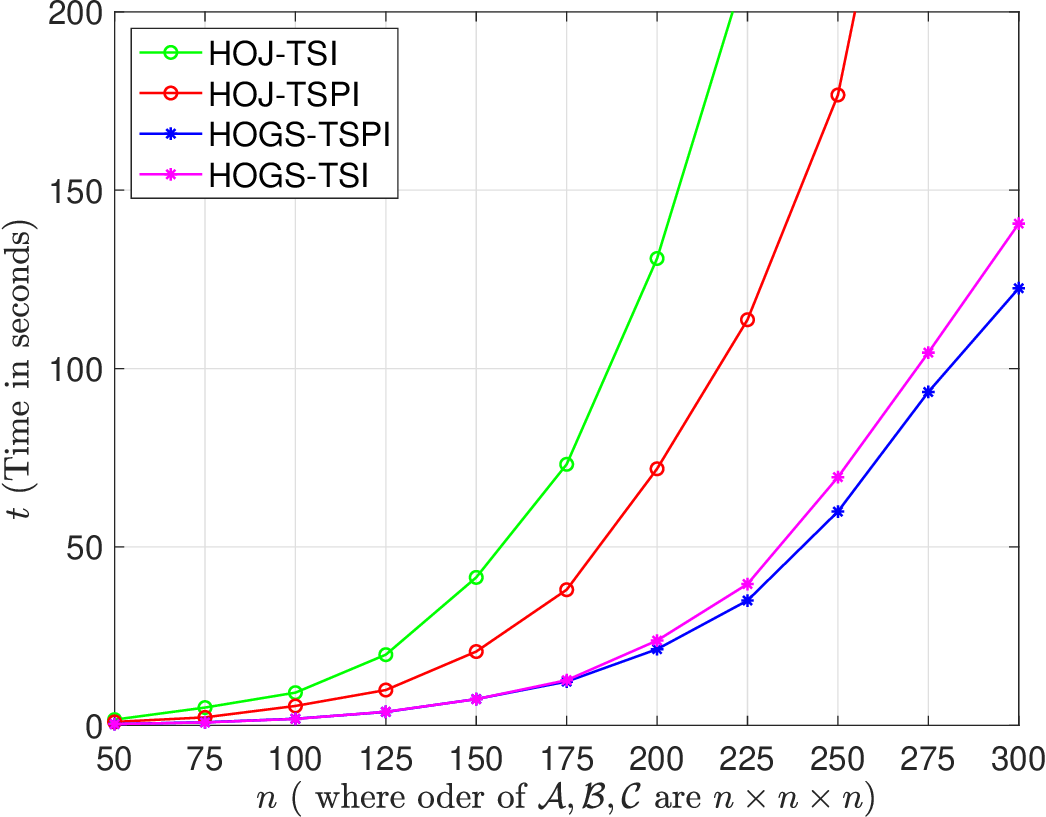}}
\subfigure[$\alpha=0.95=\beta$]{\includegraphics[scale=0.45]{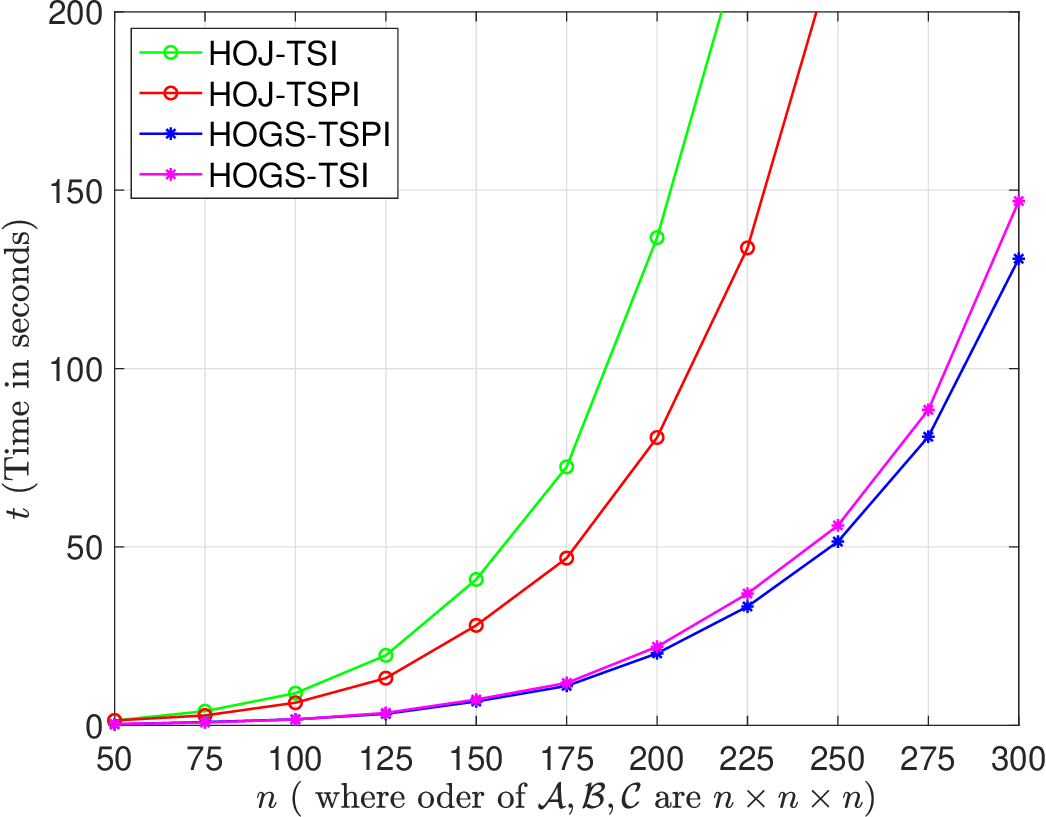}}
~~\subfigure[$\alpha=1.1=\beta$]{\includegraphics[scale=0.45]{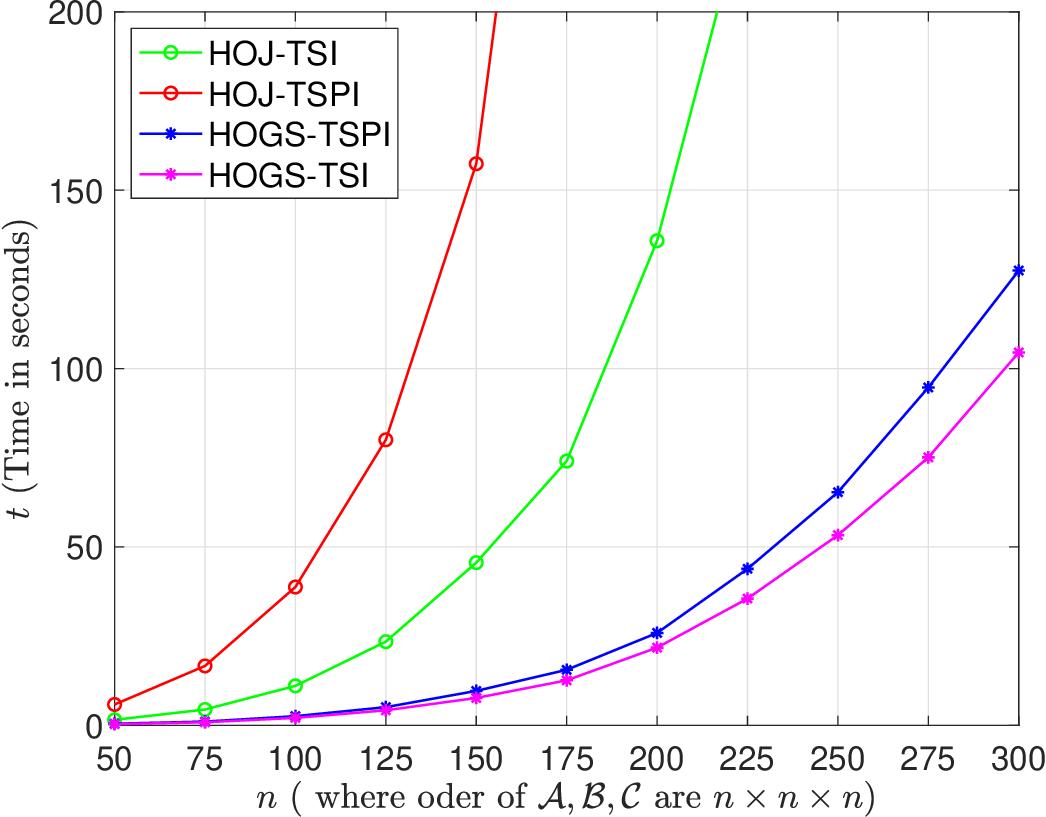}}
\caption{A comparative analysis of the average CPU times for higher-order TSI and TPSI methods, focusing on various parameters such as $\alpha$ and $\beta$}
\label{Fig2d_comp}
\end{figure}
\begin{figure}[H]
\centering
\subfigure[$n=100$]{\includegraphics[scale=0.45]{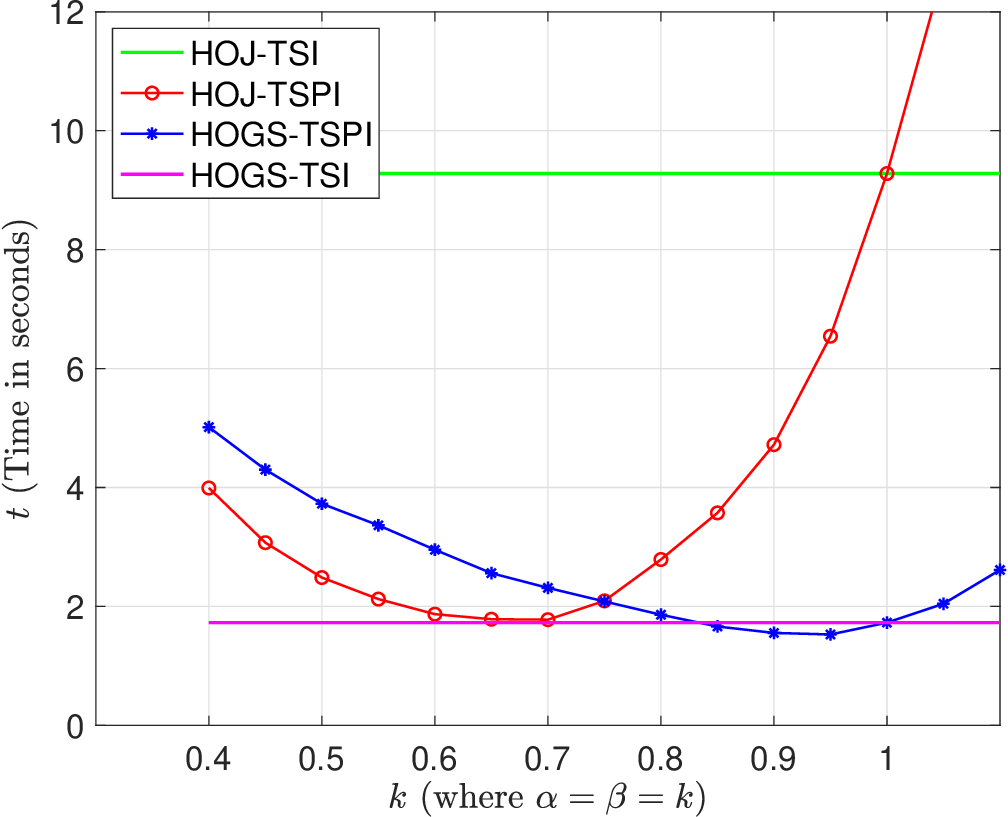}}
~~\subfigure[$n=150$]{\includegraphics[scale=0.45]{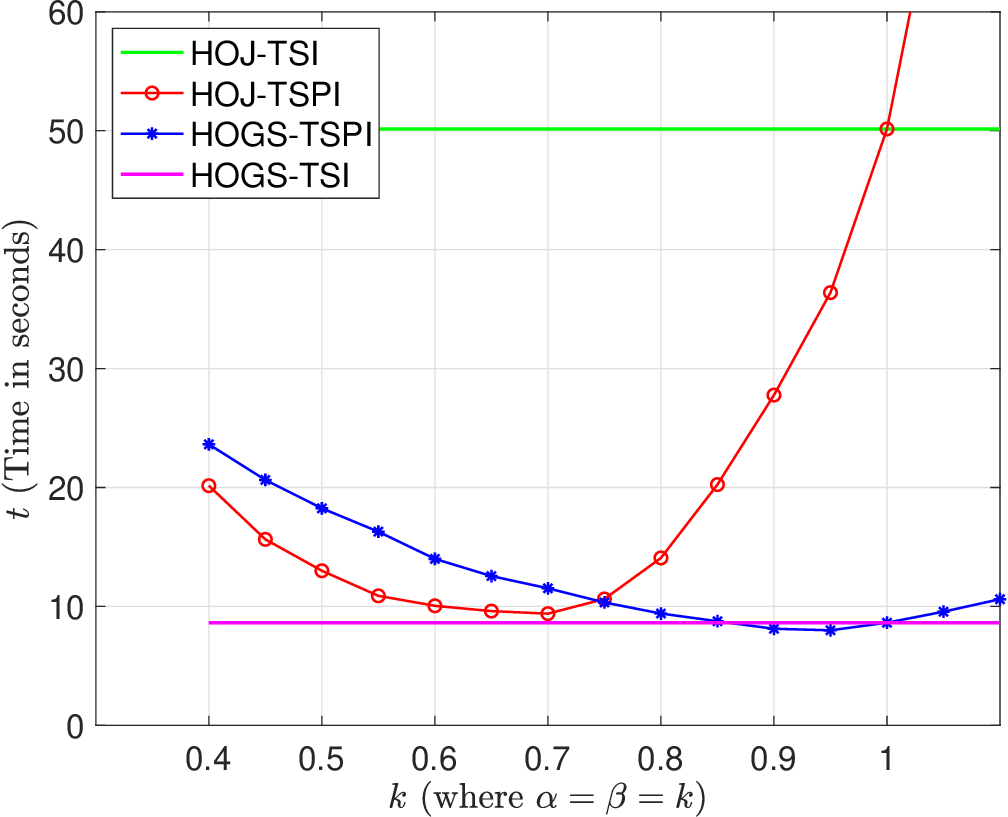}}
\subfigure[$n=200$]{\includegraphics[scale=0.45]{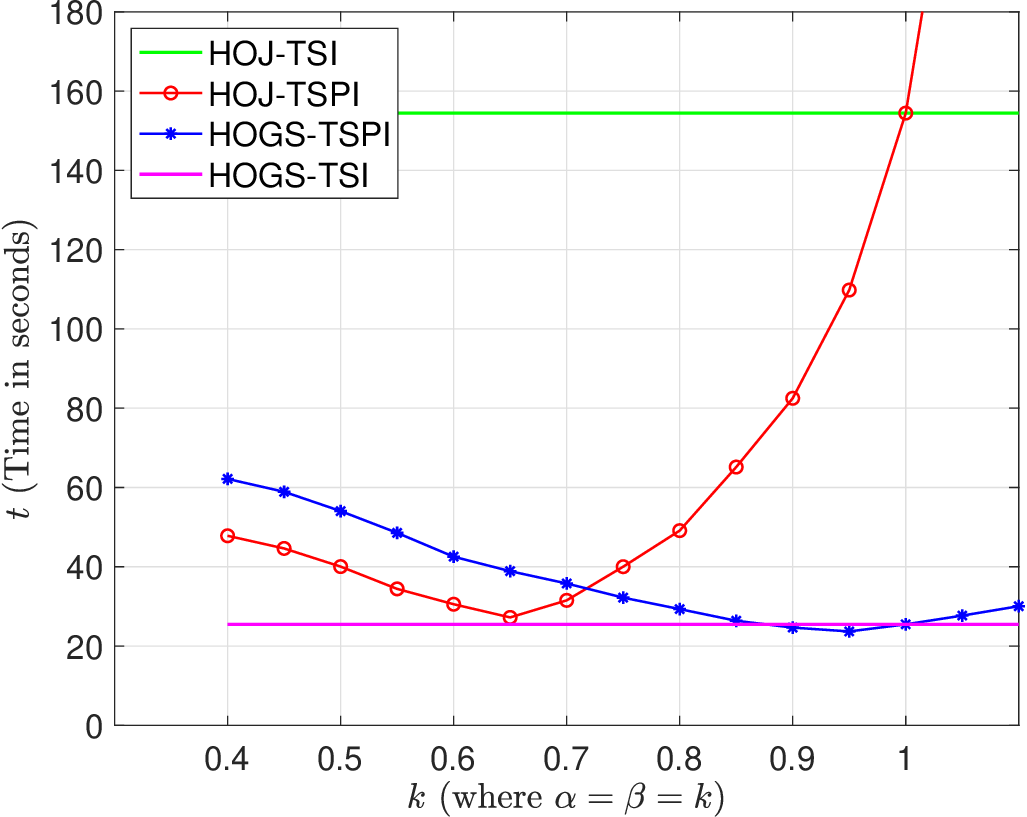}}
~~\subfigure[$n=250$]{\includegraphics[scale=0.45]{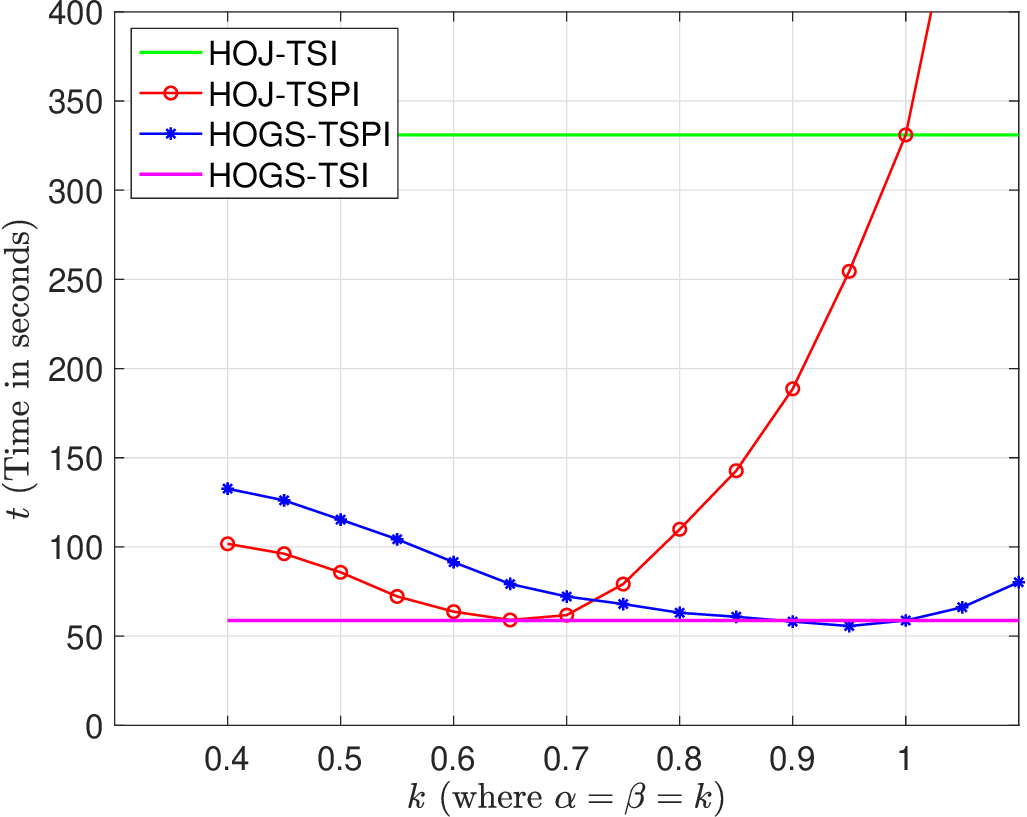}}
\caption{Comparison analysis of mean CPU time of higher order TSI and TPSI methods for different tensor sizes $n\times n\times  n$ ( where $\mc{A},~\mc{B},~\mc{C}\in\mathbb{C}^{n\times n\times n}$ ) by varying parameters $\alpha$ and $\beta$}
\label{Fig2d_alpha_comp}
\end{figure}

\begin{figure}[H]
\centering
\subfigure[HOJ-TSPI ($n=100$)]{\includegraphics[scale=0.4]{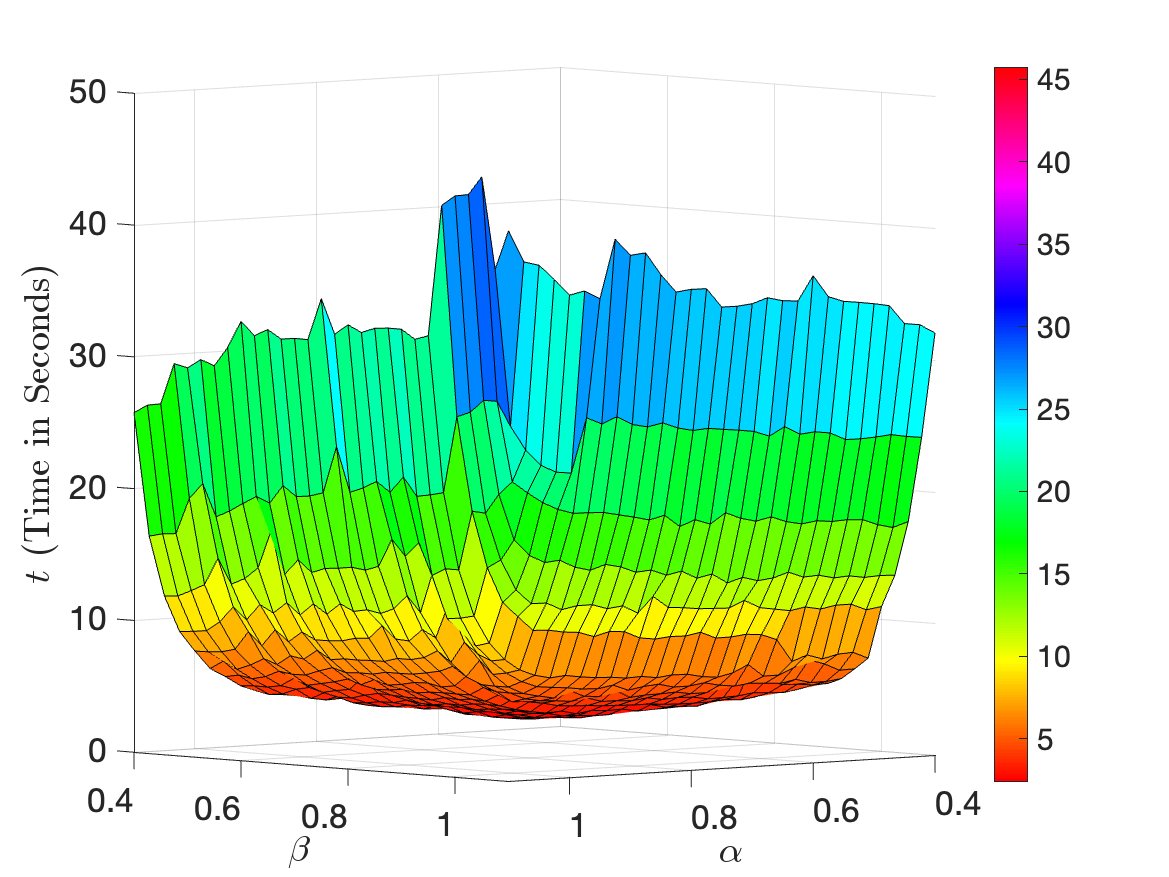}}
\subfigure[HOGS-TSPI  ($n=100$)]{\includegraphics[scale=0.4]{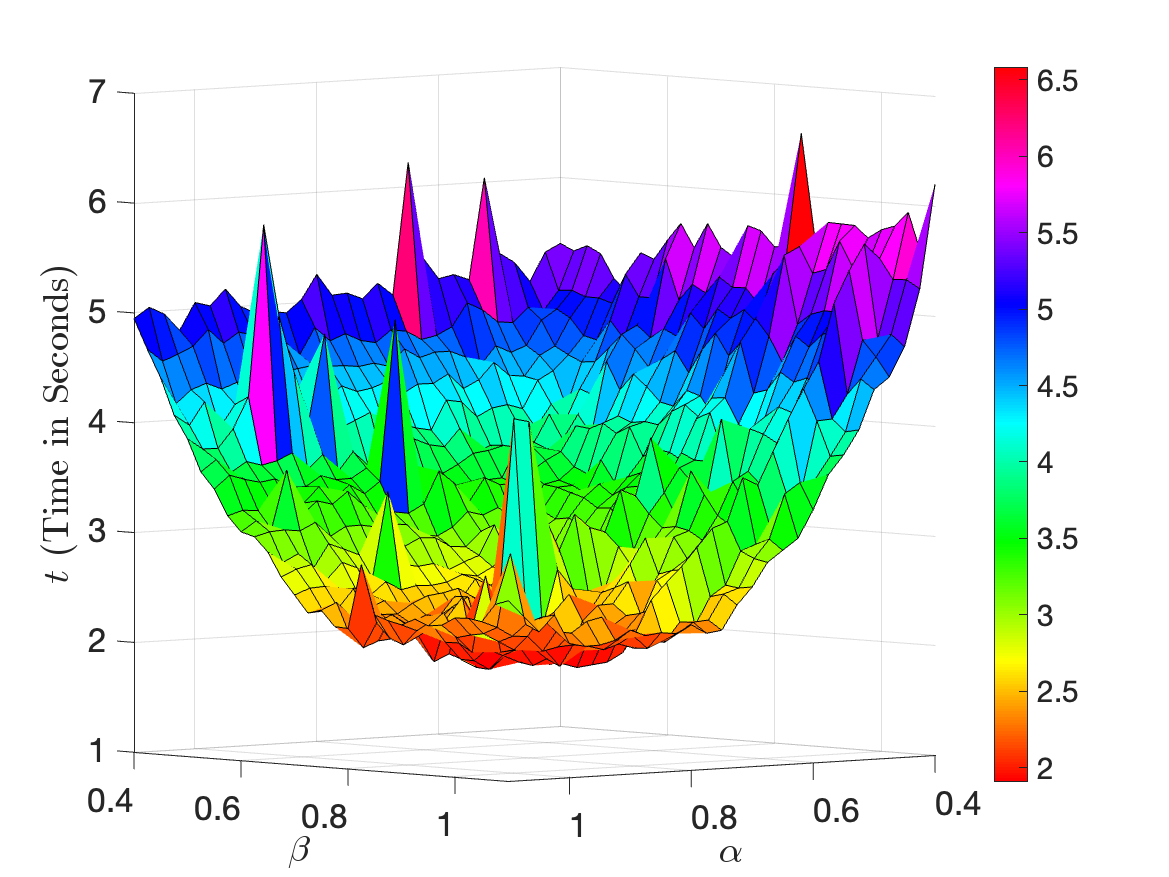}}
\subfigure[HOJ-TSPI ($n=150$)]{\includegraphics[scale=0.4]{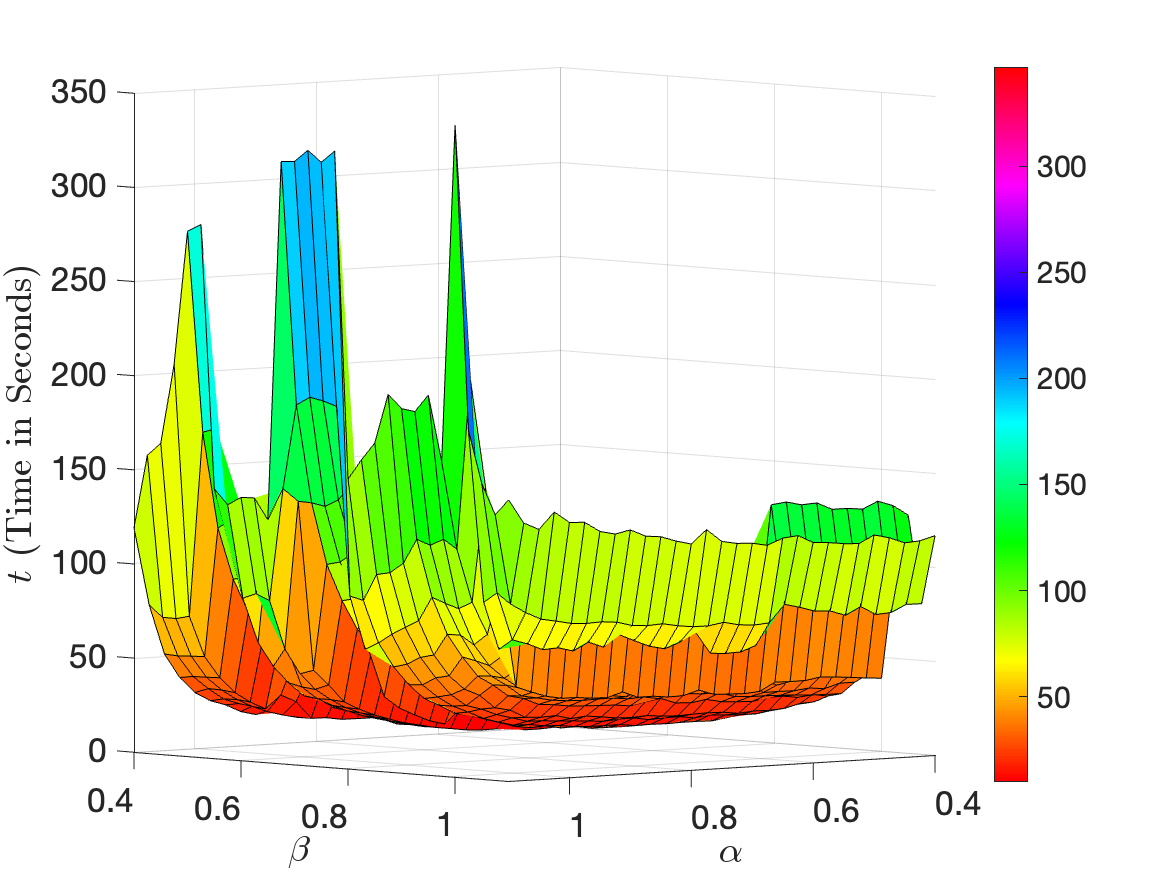}}
\subfigure[HOGS-TSPI ($n=150$)]{\includegraphics[scale=0.4]{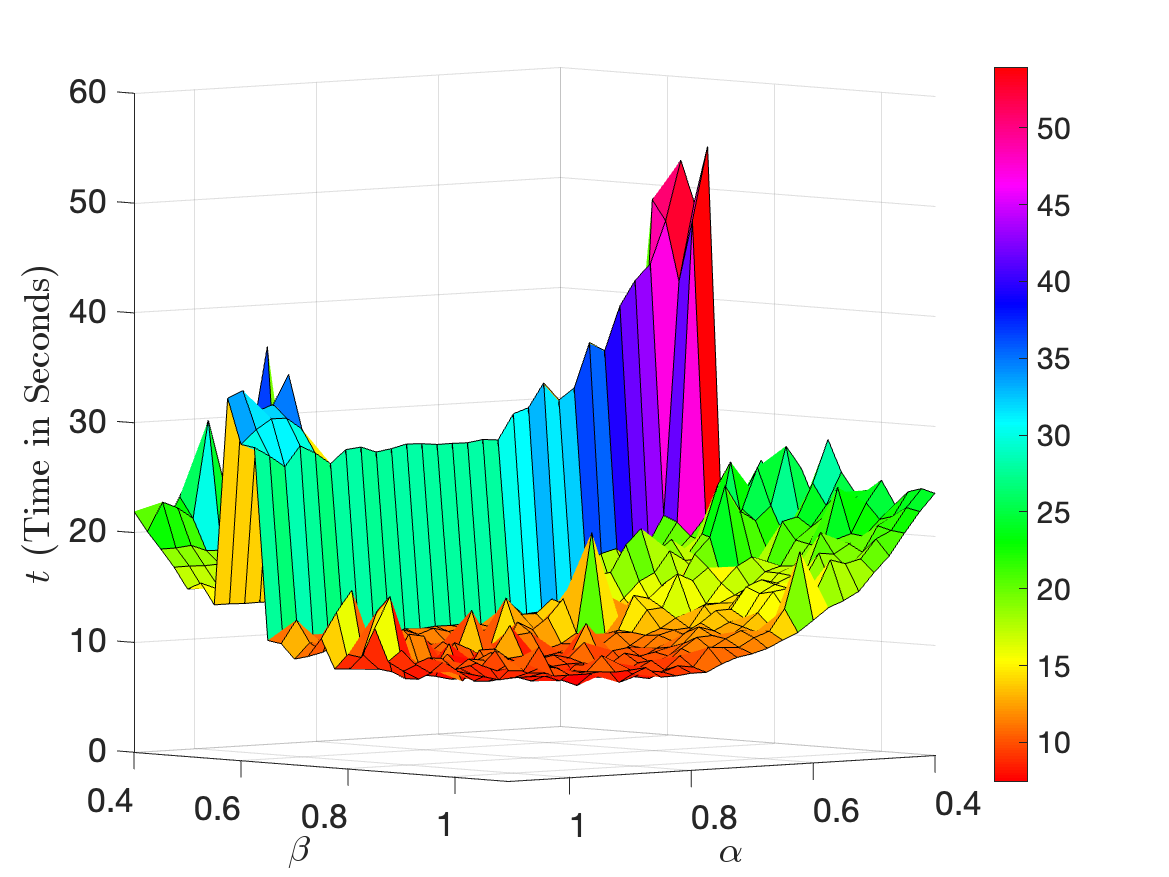}}
\caption{A comparative analysis of the average CPU time for higher-order TSPI methods  focusing on fixed tensor sizes of  $n\times n\times  n$ ( where $\mc{A},~\mc{B},~\mc{C}\in\mathbb{C}^{n\times n\times n}$ ) with varying both parameters $\alpha$ and $\beta$}
\label{Fig3d_alpha_beta_comp}
\end{figure}

\begin{theorem}
Consider $M\in\mbc^{p\times p}$, $\mc{A}$, $\mc{B}$,  $\mc{C}$ as defined in \eqref{maineq}. Let $0<\alpha<\frac{2}{1+\rho\left(\mathscr{F}_1^{-1} \m\mathscr{G}_1\right)}, ~0<\beta<\frac{2}{1+\rho\left(\mathscr{F}_2^{-1}\m \mathscr{G}_2\right)}$.  If $\mc{A}$ and $\mc{B}$ are strictly diagonally dominant tensors then the HOSOR-TSPI method converges
   for $0<\omega_1, \omega_2 \leq 1$.
\end{theorem}
\begin{theorem}
   Consider $M\in\mbc^{p\times p}$, $\mc{A}$, $\mc{B}$,  $\mc{C}$ as defined in \eqref{maineq}. Let $0<\alpha<\frac{2}{1+\rho\left(\mathscr{F}_1^{-1} \m\mathscr{G}_1\right)}, ~0<\beta<\frac{2}{1+\rho\left(\mathscr{F}_2^{-1}\m \mathscr{G}_2\right)}$.  If $\mc{A}$ and $\mc{B}$  are Hermitian positive definite tensors, then the HOSOR-TSPI method converges for $0<\omega_1, \omega_2<2.$
\end{theorem}

\begin{corollary}
 Consider $M\in\mbc^{p\times p}$, $\mc{A}$, $\mc{B}$,  $\mc{C}$ as defined in \eqref{maineq}. Let $0<\alpha<\frac{2}{1+\rho\left(\mathscr{F}_1^{-1} \m\mathscr{G}_1\right)}, ~0<\beta<\frac{2}{1+\rho\left(\mathscr{F}_2^{-1}\m \mathscr{G}_2\right)}$.  If $\mc{A}$ and $\mc{B}$  are Hermitian positive definite tensors then the HOGS-TSPI method is convergent.
\end{corollary}
Next, we discuss the nonnengative splittings of tensors under $M$-product.
\begin{definition}\label{def-reg}
   Consider $M\in\mbc^{p\times p}$ and $\mc{A}\in\mbc^{m\times m\times p}$ to be nonsingular. The splitting $\mc{A}=\mc{F}-\mc{G}$ is called a
   \begin{enumerate}
       \item[(i)] regular splitting of $\mc{A}$  if $\mc{F}^{-1}\geq 0$ and $\mc{G}\geq 0$.
       \item[(ii)] weak regular splitting of $\mc{A}$  if $\mc{F}^{-1}\geq 0$ and $\mc{F}^{-1}\m\mc{G}\geq 0$.
         \item[(iii)] nonnegative splitting of $\mc{A}$  if $\mc{F}^{-1}\m\mc{G}\geq 0$.
   \end{enumerate}
\end{definition}

\begin{theorem}\label{thm-non}
Let $M\in\mbc^{p\times p}$ and $\mc{A}\in\mbc^{m\times m\times p}$ be nonsingular. Consider $\mc{A}=\mc{F}-\mc{G}$ as a nonnegative (or weak) splitting of $\mc{A}$,  and  $\mc{T}=\mc{A}^{-1}\m\mc{F}$. Then $\mc{T}\geq 0$ if and only if $\rho(\mc{F}^{-1}\m\mc{G})<1$.
\end{theorem}
\begin{proof}
Let  $\mc{A}=\mc{F}-\mc{G}$ be a nonnegative splitting of $\mc{A}$. Then from Definition \ref{def-reg}, we have $(\wtilde{\mc{F}^{-1}\m\mc{G}})(:,:,i)\geq 0$. Using Proposition \ref{prop-basic} (i) and (iv) with $\mc{B}=\mc{I}_{mmp}$, we obtain
   \begin{equation}\label{eq13}
\wtilde{\mc{F}^{-1}\m\mc{G}}=\wtilde{\mc{F}^{-1}}\Delta\tilde{\mc{G}}=(\tilde{\mc{F}})^{-1}\Delta\tilde{\mc{G}}.  
   \end{equation}
Clearly $\tilde{\mc{A}}=\tilde{\mc{F}}-\tilde{\mc{G}}$ and from equation \eqref{eq13}, it follows that $\tilde{\mc{A}}(:,:,i)=\tilde{\mc{F}}(:,:,i)-\tilde{\mc{G}}(:,:,i)$ is a nonnegative splitting of $\tilde{\mc{A}}(:,:,i)$ for each $i$, $i=1,2,\ldots, p$. Again applying Proposition \ref{prop-basic}, we have $\tilde{\mc{T}}=\wtilde{\mc{A}^{-1}\m\mc{F}}=\wtilde{\mc{A}^{-1}}\Delta\tilde{\mc{F}}=(\tilde{\mc{A}})^{-1}\Delta\tilde{\mc{F}}$. Thus $\mc{T}\geq 0$ if and only if $(\tilde{\mc{A}})^{-1}(;,:,i)\tilde{\mc{F}}(:,:,i)\geq 0$ for all $i=1,2,\ldots,p$. By Lemma \ref{lem-non}, $(\tilde{\mc{A}})^{-1}(;,:,i)\tilde{\mc{F}}(:,:,i)\geq 0$ if only if $\rho\left((\tilde{\mc{F}})^{-1}(:,:,i)\tilde{\mc{G}}(:,:,i)\right)<1$ for all $i=1,2,\ldots,p$. Which is equivalently, $\mc{T}\geq 0$ if and only if $\rho(\mc{F}^{-1}\m\mc{G})<1$.
\end{proof}
In the similar manner by using Lemma \ref{lem-reg} and Lemma \ref{lem-weak}, we can prove the following two results. 
\begin{corollary}\label{thm-weak}
  Consider $M\in\mbc^{p\times p}$ and $\mc{A}\in\mbc^{m\times m\times p}$ to be nonsingular. Consider $\mc{A}=\mc{F}-\mc{G}$ as a weak regular splitting of $\mc{A}$. Then $\mc{A}^{-1}\geq 0$ if and only if $\rho(\mc{F}^{-1}\m\mc{G})<1$.
\end{corollary}
\begin{corollary}\label{cor-reg}
    Consider $M\in\mbc^{p\times p}$ and $\mc{A}\in\mbc^{m\times m\times p}$ to be nonsingular. Consider $\mc{A}=\mc{F}-\mc{G}$ as a regular splitting of $\mc{A}$. Then $\mc{A}^{-1}\geq 0$ if and only if $\rho(\mc{F}^{-1}\m\mc{G})<1$. 
\end{corollary}
In view of Theorem \ref{thm-cgt} and Theorem \ref{thm-non}, we obtain the following result.
\begin{theorem}\label{thm-non-para}
Consider $M\in\mbc^{p\times p}$, $\mc{A}$, $\mc{B}$,  $\mc{C}$ as defined in \eqref{maineq}. Let $\mc{A}=\mc{F}_1-\mc{G}_1$ be a nonnegative splitting of $\mc{A}$ and  $\mc{B}=\mc{F}_2-\mc{G}_2$ be a nonnegative  splitting of $\mc{A}$. Consider $\left\{\mc {X}_k\right\}$, $\left\{\mc {Y}_k\right\}$  as defined in \eqref{tpit}, $0<\alpha<\frac{2}{1+\rho(\mc{F}_{1}^{-1}\m\mc{G}_1)}$ and $0<\beta<\frac{2}{1+\rho(\mc{F}_{2}^{-1}\m\mc{G}_2)}$.  If $\mc{A}^{-1}\m\mc{F}_1\geq 0$ and $\mc{B}^{-1}\m\mc{F}_2\geq 0$ then   $\left\{\mc{X}_k\right\}$  converges  to the exact solution $\mc{A}^{-1}\m\mc{C}\m\mc{B}^{-1}$ for any initial tensor $\mc{X}_0$ with $\mc{Y}_0=\mc{X}_0\m\mc{B}$.
\end{theorem}
Following Theorem \ref{thm-cgt} along with Corollaries \ref{thm-weak} and \ref{cor-reg}, we obtain the following results.
\begin{theorem}\label{weak-reg-ten}
   Consider $M\in\mbc^{p\times p}$, $\mc{A}$, $\mc{B}$,  $\mc{C}$ as defined in \eqref{maineq}. Let $\mc{A}=\mc{F}_1-\mc{G}_1$ be a weak regular splitting of $\mc{A}$ and  $\mc{B}=\mc{F}_2-\mc{G}_2$ be a weak regular  splitting of $\mc{B}$. Consider $\left\{\mc {X}_k\right\}$, $\left\{\mc {Y}_k\right\}$  as defined in \eqref{tpit}, $0<\alpha<\frac{2}{1+\rho(\mc{F}_{1}^{-1}\m\mc{G}_1)}$ and $0<\beta<\frac{2}{1+\rho(\mc{F}_{2}^{-1}\m\mc{G}_2)}$.  If $\mc{A}^{-1}\geq 0$ and $\mc{B}^{-1}\geq 0$ then   $\left\{\mc{X}_k\right\}$  converges  to the exact solution $\mc{A}^{-1}\m\mc{C}\m\mc{B}^{-1}$ for any initial tensor $\mc{X}_0$ with $\mc{Y}_0=\mc{X}_0\m\mc{B}$. 
\end{theorem}
\begin{corollary}
    Consider $M\in\mbc^{p\times p}$, $\mc{A}$, $\mc{B}$,  $\mc{C}$ as defined in \eqref{maineq}. Let $\mc{A}=\mc{F}_1-\mc{G}_1$ be a regular splitting of $\mc{A}$ and $\mc{B}=\mc{F}_2-\mc{G}_2$ be a regular splitting of $\mc{B}$. Consider $\left\{\mc {X}_k\right\}$, $\left\{\mc {Y}_k\right\}$  as defined in \eqref{tpit}, $0<\alpha<\frac{2}{1+\rho(\mc{F}_{1}^{-1}\m\mc{G}_1)}$ and $0<\beta<\frac{2}{1+\rho(\mc{F}_{2}^{-1}\m\mc{G}_2)}$.  If $\mc{A}^{-1}\geq 0$ and $\mc{B}^{-1}\geq 0$ then   $\left\{\mc{X}_k\right\}$  converges  to the exact solution $\mc{A}^{-1}\m\mc{C}\m\mc{B}^{-1}$ for any initial tensor $\mc{X}_0$ with $\mc{Y}_0=\mc{X}_0\m\mc{B}$. 
\end{corollary}

\subsection{Preconditioned two-step parametrized iterative (PTSPI) method}
Notice that if $\mc{A}^{-1}\ngeq 0$, then Corollary \ref{thm-weak} and Corollary \ref{cor-reg} need not be true in general. To avoid such a problem, we can introduce preconditioned nonsingular tensors $\mc{P}_1\in\mbc^{m\times m\times p}$ and $\mc{P}_2\in\mbc^{n\times n\times p}$ such that $\mc{P}_1\m\mc{A}\geq 0$ and $\mc{B}\m\mc{P}_2\geq 0$, which guarantees convergence.  Based on the preconditioned tensors, the tensor equation \eqref{maineq} is modified to
\begin{equation}\label{eq14}
 \mc{P}_1 \m \mc{A} \m\mc{X}\m\mc{B}\m\mc{P}_2=\mc{P}_1 \m\mc{C}\m\mc{P}_2.
\end{equation}
Equation \eqref{eq14} reduced into the following two coupled tensor equations
\[
 \mc{P}_1 \m \mc{A} \m \mc{Y}=\mc{P}_1 \m\mc{C}\m\mc{P}_2 \text{~and~}
\mc{X}\m\mc{B}\m\mc{P}_2=\mc{Y} .
\] 
Let $\mc{P}_1 \m \mc{A} =\mc{F}_{p_1}-\mc{G}_{p_1} \text{~and~}  \mc{B}\m\mc{P}_2=\mc{F}_{p_2}-\mc{G}_{p_2}$, 
where $\mc{F}_{p_1}$ and $\mc{F}_{p_2}$ are nonsingular tensors. Similar to \eqref{tpit}, we obtain the following PTPSI method for solving equation \eqref{maineq}.
\begin{equation}\label{ptpit}
\left\{\begin{array}{ll}\mc{Y}_{k+1}&=\left(\mc{I}_{mmp}-\alpha\mc{F}_{p_1}^{-1}\m\mc{P}_1 \m \mc{A} \right)\m\mc{Y}_k+\alpha\mc{F}_{p_1}^{-1}\m\mc{P}_1\m\mc{C}\m\mc{P}_2\\
\mc{X}_{k+1}&=\mc{X}_k\m\left(\mc{I}_{nnp}-\beta\mc{B}\m\mc{P}_2\m\mc{F}_{p_2}^{-1} \right)+ \beta\mc{Y}_{k+1}\m\mc{F}_{p_2}^{-1} 
\end{array}
\right.    
\end{equation}
The next theorem follows directly from Theorem \ref{thm-cgt}.
\begin{theorem}\label{thm-pcgt}
Consider $M\in\mbc^{p\times p}$, $\mc{A}$, $\mc{B}$,  $\mc{C}$, $\mc{P}_1$ and $\mc{P}_2$ as defined in \eqref{eq14}. Let $\left\{\mc {X}_k\right\}$ and   $\left\{\mc {Y}_k\right\}$  be defined as in \eqref{ptpit}.  Assume that $0<\alpha<\frac{2}{1+\rho(\mc{F}_{p_1}^{-1}\m\mc{G}_{p_1})}$ and $0<\beta<\frac{2}{1+\rho(\mc{F}_{p_2}^{-1}\m\mc{G}_{p_2})}$.  If $\rho(\mc{F}_{p_1}^{-1}\m\mc{G}_{p_1})<1$ and $\rho(\mc{F}_{p_2}^{-1}\m\mc{G}_{p_2})<1$ then   $\left\{\mc{X}_k\right\}$  converges  to the exact solution 
\[(\mc{P}_1 \m \mc{A})^{-1}\m\mc{P}_1 \m\mc{C}\m\mc{P}_2\m (\mc{B}\m\mc{P}_2)^{-1}=\mc{A}^{-1}\m\mc{C}\m\mc{B}^{-1},\]
for any initial tensor $\mc{X}_0$ with $\mc{Y}_0=\mc{X}_0\m\mc{B}$.
\end{theorem}
Next, we design an algorithm to implement the preconditioned two-step parameterized iterative scheme.
\begin{algorithm}[H]
 \caption{ PTSPI method for solving $\mc{A}\m\mc{X}\m\mc{B}=\mc{C}$} \label{alg:PTSPI}
\begin{algorithmic}[1]
\State {\bf Input} $\mc{A}\in \mathbb{C}^{m \times m\times p}$, $\mc{B} \in \mathbb{C}^{n \times n\times p}$,  $\mc{C} \in \mathbb{C}^{m \times n\times p}$,   $M\in\mbc^{p\times p}$ $\alpha,~\beta$.
\State {\bf Initial guess }$\mc{X}_0 \in \mathbb{C}^{m \times n\times p}$ and $\mc{Y}_0=\mc{X}_0\m\mc{B}$
\State {\bf Preconditioned tensors }$\mc{P}_1,~\mc{P}_2$
\State {\bf Compute } Weak regular splittings of $\mc{P}_1\m\mc{A}=\mc{F}_{p_1}-\mc{G}_{p_1}$ and $\mc{B}\m\mc{P}_2=\mc{F}_{p_2}-\mc{G}_{p_2}$
\While{({true})}
\State $\mc{Y}_{k+1}=\left(\mc{I}_{mmp}-\alpha\mc{F}_{p_1}^{-1}\m\mc{P}_1 \m \mc{A} \right)\m\mc{Y}_k+\alpha\mc{F}_{p_1}^{-1}\m\mc{P}_1\m\mc{C}\m\mc{P}_2$
\State $ \mc{X}_{k+1}=\mc{X}_k\m\left(\mc{I}_{nnp}-\beta\mc{B}\m\mc{P}_2\m\mc{F}_{p_2}^{-1} \right)+ \beta\mc{Y}_{k+1}\m\mc{F}_{p_2}^{-1}$. 
   \If{tol $\leq \epsilon$}
   \State\textbf{break} 
  \EndIf
  \State $\mc{Y}_0\gets \mc{Y}_k$ and $\mc{X}_0\gets \mc{X}_k$
   \EndWhile 
\State \Return $\mc{X}_k$ 
 \end{algorithmic}
\end{algorithm}
Theorem \ref{weak-reg-ten} can be restated for preconditioned tensors as follows.
\begin{theorem}
   Consider $M\in\mbc^{p\times p}$, $\mc{A}$, $\mc{B}$,  $\mc{C}$, $\mc{P}_1$ and $\mc{P}_2$ as defined in \eqref{eq14}.  Let $\mc{P}_1\m\mc{A}=\mc{F}_{p_1}-\mc{G}_{p_1}$ be a weak regular splitting of $\mc{P}_1\m\mc{A}$ and  $\mc{B}\m\mc{P}_2=\mc{F}_{p_2}-\mc{G}_{p_2}$ be a weak regular  splitting of $\mc{B}\m\mc{P}_2$. Consider $\left\{\mc {X}_k\right\}$, $\left\{\mc {Y}_k\right\}$  as defined in \eqref{ptpit},  $0<\alpha<\frac{2}{1+\rho(\mc{F}_{p_1}^{-1}\m\mc{G}_{p_1})}$ and $0<\beta<\frac{2}{1+\rho(\mc{F}_{p_2}^{-1}\m\mc{G}_{p_2})}$.  If $(\mc{P}_1\m\mc{A})^{-1}\geq 0$ and $(\mc{B}\m\mc{P}_2)^{-1}\geq 0$ then   $\left\{\mc{X}_k\right\}$  converges  to the exact solution $\mc{A}^{-1}\m\mc{C}\m\mc{B}^{-1}$ for any initial tensor $\mc{X}_0$ with $\mc{Y}_0=\mc{X}_0\m\mc{B}$. 
\end{theorem}

\begin{theorem}
Consider $M\in\mbc^{p\times p}$, $\mc{A}$, $\mc{B}$,  $\mc{C}$, $\mc{P}_1$ and $\mc{P}_2$ as defined in \eqref{eq14} with $\mc{A}^{-1}\geq 0,~\mc{B}^{-1}\geq 0,~(\mc{P}_1\m\mc{A})^{-1}\geq 0$ and $(\mc{B}\m\mc{P}_2)^{-1}\geq 0$.  Let $\mc{A}=\mc{F}_1-\mc{G}_1$  $\mc{P}_1\m\mc{A}=\mc{F}_{p_1}-\mc{G}_{p_1}$, $\mc{B}=\mc{F}_2-\mc{G}_2$ and $\mc{B}\m\mc{P}_2=\mc{F}_{p_2}-\mc{G}_{p_2}$ be respectively a weak regular splitting of $\mc{A}$, $\mc{P}_1\m\mc{A}$, $\mc{B}$ and  $\mc{B}\m\mc{P}_2$. If
\[\rho(\mc{F}_{p_1}^{-1}\m\mc{G}_{p_1})<\rho(\mc{F}_1^{-1}\m\mc{G}_1),~~\rho(\tilde{\mc{F}_2}\m\tilde{\mc{G}_2})<\rho(\mc{F}_2^{-1}\m\mc{G}_2),\]
then PTSPI method \eqref{ptpit}  converges faster than iterative scheme \eqref{tpit} for $0<\alpha<1$ and $0<\beta<1$.
\end{theorem}
\begin{proof}
Let $0<\alpha<1$ and $0<\beta<1$. For each $i$, since $\wtilde{(\mc{F}_1^{-1}\m\mc{G})}(:,:,i)$ is a nonnegative matrix, so by Lemma \ref{preron}, there exist a nonnegative eigenvalue (say $\lambda$ ) such that $\rho\left(\wtilde{(\mc{F}_1^{-1}\m\mc{G})}(:,:,i)\right)=\lambda$. Let $\mc{T}=\mc{I}_{mmp}-\alpha\mc{F}_{1}^{-1}\m\mc{A}=(1-\alpha)\mc{I}_{mmp}+\alpha\mc{F}_{1}^{-1}\m\mc{G}$. Then 
\begin{equation}\label{eq16}
1-\alpha + \alpha\rho\left(\wtilde{(\mc{F}_1^{-1}\m\mc{G})}(:,:,i)\right)= \rho (\tilde{T}(:,:,i)),~i=1,2,\ldots,p.   
\end{equation}
From equation \eqref{eq16}, we obtain 
\begin{eqnarray}\label{eq17}
\nonumber
   \rho(\mc{I}_{mmp}-\alpha\mc{F}_{1}^{-1}\m\mc{A} )&=&\max_{1\leq i\leq p}\rho (\tilde{T}(:,:,i))=1-\alpha+\alpha\max_{1\leq i\leq p}\rho\left(\wtilde{(\mc{F}_1^{-1}\m\mc{G})}(:,:,i)\right)\\
   &=&1-\alpha+\alpha\rho(\mc{F}_{1}^{-1}\m\mc{G}).
\end{eqnarray}
In the similar manner, for the splitting $\mc{P}_1\m\mc{A}=\mc{F}_{p_1}-\mc{G}_{p_1}$, we can show that 
\begin{equation}\label{eq18}
 \rho (\mc{I}_{mmp}-\alpha\mc{F}_{p_1}^{-1}\m\mc{P}_1\m\mc{A})= 1-\alpha + \alpha\rho(\mc{F}_{p_1}^{-1}\m \mc{G}_{p_1}).   
\end{equation}
 From equations \eqref{eq17}, \eqref{eq18} and the given condition $\rho(\mc{F}_{p_1}^{-1}\m\mc{G}_{p_1})<\rho(\mc{F}_1^{-1}\m\mc{G}_1)$, we get  
 \begin{equation}\label{eq19}
   \rho (\mc{I}_{mmp}-\alpha\mc{F}_{p_1}^{-1}\m\mc{P}_1\m\mc{A})<\rho(\mc{I}_{mmp}-\alpha\mc{F}_{1}^{-1}\m\mc{A} ).  
 \end{equation}
Similarly, we can show that 
\begin{equation}\label{eq20}
   \rho (\mc{I}_{nnp}-\alpha\mc{F}_{p_2}^{-1}\m\mc{P}_2\m\mc{B})<\rho(\mc{I}_{nnp}-\alpha\mc{F}_{2}^{-1}\m\mc{B} ).  
 \end{equation}
\end{proof}
The proof is complete from equation \eqref{eq19} and equation \eqref{eq20}.
\begin{example}\rm\label{ex-pcond-1}
Let $M=\begin{bmatrix}
    1&0\\
    0&2
\end{bmatrix}$  and consider a multilinear system $\mc{A}\m\mc{X}\m\mc{B}=\mc{C}$ with 
 \[\mc{A}(:,:,1)=\begin{bmatrix}
     2   &          -1      &       -1  \\     
	      -2      &        7/2      &     -1/2 \\    
	      -5/2       &    -3/2     &       7/2 
 \end{bmatrix},~\mc{A}(:,:,2)=\begin{bmatrix}
     2        &     -2       &       0     \\  
	      -3       &       5     &        -2\\       
	      -1            & -2       &       9 
 \end{bmatrix},~\mc{C}(:,:,1)=\begin{bmatrix}
     1  &          1      &       1  \\     
	      0     &        1      &     -1\\    
	      2       &    -3     &       0 
 \end{bmatrix}\]
  \[\mc{B}(:,:,1)=\begin{bmatrix}
    5/2      &     -1     &         0    \\   
	      -1          &    5/2     &      -1   \\    
	      -3/2       &    -2     &         2 
 \end{bmatrix},~\mc{B}(:,:,2)=\begin{bmatrix}
   5          &   -2     &         0       \\
	      -2        &      5     &        -2  \\     
	      -3          &   -4     &         4 
 \end{bmatrix},~ \mc{C}(:,:,2)=\begin{bmatrix}
     -1   &          -1      &       -1  \\     
	      0      &        2      &     -1 \\    
	      0       &    -3    &       0 
 \end{bmatrix}.\]
We can verify that $\mc{Z}_{\mc{A}}:=\mc{A}^{-1}\geq 0$ and $\mc{Z}_{\mc{B}}:=\mc{B}^{-1}\geq 0$  since 
\[\wtilde{~\mc{Z}_{\mc{A}}}(:,:,1)=\begin{bmatrix}
  23/6     &       5/3      &      4/3     \\
	      11/4        &    3/2       &     1  \\     
	      47/12     &     11/6   &         5/3    
\end{bmatrix},~\wtilde{~\mc{Z}_{\mc{A}}}(:,:,2)=\begin{bmatrix}
 41/48       &    3/8      &      1/12  \\  
	      29/48      &     3/8      &      1/12 \\   
	      11/48       &    1/8     &       1/12  
\end{bmatrix}\mbox{ and}\]
 \[\wtilde{~\mc{Z}_{\mc{B}}}(:,:,1)=\begin{bmatrix}
 3/4   &         1/2    &        1/4     \\
	       7/8       &     5/4      &      5/8 \\    
	      23/16      &    13/8       &    21/16 
\end{bmatrix},~\wtilde{~\mc{Z}_{\mc{B}}}(:,:,2)=\begin{bmatrix}
  3/16      &     1/8     &       1/16  \\  
	       7/32    &       5/16  &         5/32 \\   
	      23/64        &  13/32   &       21/64
\end{bmatrix}.\]
Consider the tensor splittings for $\mc{A}$ and $\mc{B}$ as  $\mc{A}=\mc{F}_1-\mc{G}_1$ and $\mc{B}=\mc{F}_2-\mc{G}_2$, where 
\[\mc{F}_1(:,:,1)=\begin{bmatrix}
    2 &0 &0\\
    0& 7/2& 0\\
    0&0&7/2
\end{bmatrix},~ 
	\mc{F}_1(:,:,2) =\begin{bmatrix}
    2 &0 &0\\ 
	0& 5& 0\\
    0&0&9
    \end{bmatrix},~\mc{G}_1=\mc{F}_1-\mc{A},\]
    \[\mc{F}_2(:,:,1)=\begin{bmatrix}
    5/2 &0 &0\\
    0& 5/2& 0\\
    0&0&2
\end{bmatrix},~ 
	\mc{F}_2(:,:,2) =\begin{bmatrix}
    5 &0 &0\\ 
	0& 5& 0\\
    0&0&4
    \end{bmatrix},~\mc{G}_2=\mc{F}_2-\mc{B}.\]
It can be verified that $\mc{F}_1^{-1}\geq 0,~\mc{F}_2^{-1}\geq 0,~(\mc{F}_1^{-1}\m\mc{G}_1)\geq 0$, and $(\mc{F}_2^{-1}\m\mc{G}_2)\geq 0$. Thus both $\mc{A}=\mc{F}_1-\mc{G}_1$ and $\mc{B}=\mc{F}_2-\mc{G}_2$, respectively are weak regular splittings of $\mc{A}$ and $\mc{B}$. Let us consider the preconditioned nonsingular tensors $\mc{P}_1=\mc{P}_2=\mc{P}$, where 
\[\mc{P}(:,:,1) = \begin{bmatrix}
	       1         &     0     &         0   \\    
	       1/2         &   1    &          0  \\     
	       5/4         &   1/2       &     1 
           \end{bmatrix},~
	\mc{P}(:,:,2) = \begin{bmatrix}
	       1/4        &    0      &        0 \\      
	       1/8       &     1/4     &       0 \\      
	       5/16        &   1/8    &        1/4
            \end{bmatrix}.\]
Clearly,  $\mc{Z}_{\mc{PA}}:=(\mc{P}\m\mc{A})^{-1}\geq 0$ and $\mc{Z}_{\mc{BP}}:=(\mc{B}\m\mc{P})^{-1}\geq 0$  since
\[\wtilde{~\mc{Z}_{\mc{PA}}}(:,:,1)=\begin{bmatrix}
  5/3        &     1     &          4/3     \\
	       1          &     1        &       1  \\     
	       4/3          &   1     &          5/3 
\end{bmatrix},~\wtilde{~\mc{Z}_{\mc{PA}}}(:,:,2)=\begin{bmatrix}
7/6      &       2/3     &        1/6     \\
	       2/3      &       2/3      &       1/6  \\   
	       1/6       &      1/6       &      1/6 
\end{bmatrix}\mbox{ and}\]
 \[\wtilde{~\mc{Z}_{\mc{BP}}}(:,:,1)=\begin{bmatrix}
 3/4     &       1/2 &           1/4     \\
	       1/2     &       1     &         1/2   \\  
	       1/4          &  1/2      &      3/4 
\end{bmatrix},~\wtilde{~\mc{Z}_{\mc{BP}}}(:,:,2)=\begin{bmatrix}
 3/8        &    1/4    &        1/8    \\ 
	       1/4    &        1/2  &          1/4 \\    
	       1/8         &   1/4      &      3/8 
\end{bmatrix}.\]
Consider splittings for $\mc{P}\m\mc{A}$ and $\mc{B}\m\mc{P}$ as  $\mc{P}\m\mc{A}=\mc{F}_{p_1}-\mc{G}_{p_1}$ and $\mc{B}\m\mc{P}=\mc{F}_{p_2}-\mc{G}_{p_2}$, where 
\[\mc{F}_{p_1}(:,:,1)=\begin{bmatrix}
    2 &0 &0\\
    0& 3& 0\\
    0&0&3
\end{bmatrix},~ 
	\mc{F}_{p_1}(:,:,2) =\begin{bmatrix}
    1 &0 &0\\ 
	0& 2& 0\\
    0&0&4
    \end{bmatrix},~\mc{G}_{p_1}=\mc{F}_{p_1}-\mc{P}\m\mc{A},\]
    \[\mc{F}_{p_2}(:,:,1)=\begin{bmatrix}
    2 &0 &0\\
    0& 2& 0\\
    0&0&2
\end{bmatrix}= 
	\mc{F}_{p_2}(:,:,2),~\mc{G}_{p_2}=\mc{F}_{p_2}-\mc{B}\m\mc{P}.\]
Moreover,  we can verify that $\mc{F}_{p_1}\geq 0,~\mc{F}_{p_2}^{-1}\geq 0,~(\mc{F}_{p_1}^{-1}\m\mc{G}_{p_1})\geq 0$, and $(\mc{F}_{p_2}^{-1}\m\mc{G}_{p_2})\geq 0$. Thus both $\mc{P}\m\mc{A}=\mc{F}_{p_1}-\mc{G}_{p_1}$ and $\mc{B}\m\mc{P}=\mc{F}_{p_2}-\mc{G}_{p_2}$, respectively are weak regular splittings of $\mc{P}\m\mc{A}$ and $\mc{B}\m\mc{P}$. In addition, we evaluate 
\[\rho(\mc{F}_{p_1}^{-1}\m\mc{G}_{p_1})=0.8792<0.9424=\rho(\mc{F}_1^{-1}\m\mc{G}_1),~\rho(\tilde{\mc{F}_2}\m\tilde{\mc{G}_2})=0.7071<0.8385=\rho(\mc{F}_2^{-1}\m\mc{G}_2).\]
\end{example}
In Example \ref{ex-pcond-1}, one can notice that the preconditioned approach not only convergent but also converges faster. The comparison of residual errors, mean CPU time for different tolerance ($\epsilon$) and parameters ($\alpha,~\beta$)  are provided in Table \ref{tab-precond-1}.
\begin{table}
    \centering
    \caption{Comparison analysis for Example \ref{ex-pcond-1}}
    \vspace{0.2cm}
    \renewcommand{\arraystretch}{1.2}
    \begin{tabular}{c|c|c|c|c|c|c}
    \hline
 $\alpha,\beta$&   Splittings& $\epsilon$& IT& \scriptsize{$\|\mc{C}-\mc{A}\m\mc{X}_{k}\m\mc{B}\|$} & \scriptsize{$\|\mc{X}_{k}-\mc{A}^{-1}\m\mc{C}\m\mc{B}^{-1}\|$}& MT\\
 \hline 
$\alpha=0.95=\beta$ &   $\mc{F}_i-\mc{G}_i,~i=1,2$  & $10^{-7}$ &  297 & $1.3593e^{-07}$ & $1.6350e^{-06}$ & 0.20\\
 
$\alpha=0.95=\beta$ &   $\mc{F}_{p_i}-\mc{G}_{p_i}~i=1,2$  & $10^{-7}$ &  122 & $3.1105e^{-07}$ & $1.83170e^{-06}$ & 0.06\\
  \hline
$\alpha=0.95=\beta$ &   $\mc{F}_i-\mc{G}_i,~i=1,2$  & $10^{-9}$ &  378 & $1.4267e^{-09}$ & $1.7160e^{-08}$ & 0.25\\

$\alpha=0.95=\beta$ &   $\mc{F}_{p_i}-\mc{G}_{p_i}~i=1,2$  & $10^{-9}$ &  160 & $3.0217e^{-09}$ & $1.8709e^{-08}$ & 0.08\\
\hline   
    $\alpha=0.95=\beta$ &   $\mc{F}_i-\mc{G}_i,~i=1,2$  & $10^{-15}$ &  612 & $1.9850e^{-14}$ & $3.6171e^{-14}$ & 0.35\\

$\alpha=0.95=\beta$ &   $\mc{F}_{p_i}-\mc{G}_{p_i}~i=1,2$  & $10^{-15}$ &  272 & $2.0677e^{-14}$ & $2.9366e^{-14}$ & 0.12\\
    \hline
    $\alpha=0.6=\beta$ &   $\mc{F}_i-\mc{G}_i,~i=1,2$  & $10^{-7}$ &  462 & $2.2552e^{-07}$ & $2.7126e^{-06}$ & 0.28\\
 
$\alpha=0.6=\beta$ &   $\mc{F}_{p_i}-\mc{G}_{p_i}~i=1,2$  & $10^{-7}$ &  186 & $5.0904e^{-07}$ & $3.2870e^{-06}$ & 0.10\\
    \hline
$\alpha=0.6=\beta$ &   $\mc{F}_i-\mc{G}_i,~i=1,2$  & $10^{-9}$ &  593 & $2.2532e^{-09}$ & $2.7102e^{-08}$ & 0.34\\

$\alpha=0.6=\beta$ &   $\mc{F}_{p_i}-\mc{G}_{p_i}~i=1,2$  & $10^{-9}$ &  217& $5.1598e^{-09}$ & $3.2999e^{-08}$ & 0.11\\
   \hline
    $\alpha=0.6=\beta$ &   $\mc{F}_i-\mc{G}_i,~i=1,2$  & $10^{-15}$ &  972 & $2.5362e^{-14}$ & $8.0569e^{-14}$ & 0.51\\
 
$\alpha=0.6=\beta$ &   $\mc{F}_{p_i}-\mc{G}_{p_i}~i=1,2$  & $10^{-15}$ &  431 & $1.7751e^{-14}$ & $7.5738e^{-14}$ & 0.20\\
    \hline
     $\alpha=0.3=\beta$ &   $\mc{F}_i-\mc{G}_i,~i=1,2$  & $10^{-7}$ &  892 & $4.7073e^{-07}$ & $5.6620e^{-06}$ & 0.47\\
 
$\alpha=0.3=\beta$ &   $\mc{F}_{p_i}-\mc{G}_{p_i}~i=1,2$  & $10^{-7}$ &  341 & $1.1433e^{-06}$ & $7.4158e^{-06}$ & 0.15\\
    \hline
$\alpha=0.3=\beta$ &   $\mc{F}_i-\mc{G}_i,~i=1,2$  & $10^{-9}$ &  1157 & $4.6481e^{-09}$ & $5.5908e^{-08}$ & 0.59\\

$\alpha=0.3=\beta$ &   $\mc{F}_{p_i}-\mc{G}_{p_i}~i=1,2$  & $10^{-9}$ &  465& $1.1701e^{-08}$ & $7.4136e^{-08}$ & 0.20\\
   \hline
    $\alpha=0.3=\beta$ &   $\mc{F}_i-\mc{G}_i,~i=1,2$  & $10^{-15}$ &  1885 & $6.3205e^{-14}$ & $2.8073e^{-13}$ & 0.89\\
 
$\alpha=0.3=\beta$ &   $\mc{F}_{p_i}-\mc{G}_{p_i}~i=1,2$  & $10^{-15}$ &  834 & $5.1211e^{-14}$ & $2.8370e^{-13}$ & 0.36\\
    \hline
    \end{tabular}
    \label{tab-precond-1}
\end{table}


\section{Application to Sylvester equations and image deblurring}\label{sec4}
The Sylvester matrix equation $AY + YB^T =C$ plays a fundamental role in systems and control theory, with numerous applications in various domains \cite{ding2005gradient, DingTong05, Golub79}. Recent research has extended this classical problem to higher-order tensor frameworks, mainly focusing on the Sylvester tensor equations involving Tucker products \cite{ShiWei13}, Einsten products \cite{Baodong}, and $t$-products \cite{chen2023perturbations}. The authors of \cite{ShiWei13} comprehensively analyze the backward error and perturbation bounds for the higher Sylvester tensor equations. In addition, Lai et al. \cite{ChenLu12} developed an innovative approach that combined projection methods with Kronecker product preconditioning to solve Sylvester tensor equations. The Sylvester tensor equation under the $M$-product is represented as 
\[\mc{A}_1\m\mc{Y}+\mc{Y}\m\mc{B}_1=\mc{C}_1, \]
where $\mc{A}_1\in\mbc^{m\times m\times p}$, $\mc{B}_1\in\mbc^{n\times n\times p}$, $\mc{C}_1\in\mbc^{m\times n\times p}$, and $\mc{Y}\in\mbc^{m\times n\times p}$. It can be expressed  in block-tensor form as 
\[
\begin{bmatrix}
        \mc{A}_1&\mc{I}_{mmp}
    \end{bmatrix}\m\begin{bmatrix}
        \mc{Y}&\mc{O}\\
        \mc{O}&\mc{Y}
    \end{bmatrix}\m\begin{bmatrix}
        \mc{I}_{nnp}\\
        \mc{B}_1
    \end{bmatrix}=\mc{C}_1.
\]
In general, we can consider Sylvester tensor equation of the  following form:
\begin{equation}\label{eq21}
    \mc{A}\m\mc{X}\m\mc{B}=\mc{C},
\end{equation}
where $\mc{A}\in\mbc^{m\times k\times p}$, $\mc{X}\in\mbc^{k\times s\times p}$, $\mc{B}\in\mbc^{s\times n\times p}$ and $\mc{C}\in\mbc^{m\times n\times p}$. Next, we discuss the solution of the tensor equation \eqref{eq21}.
\begin{theorem}
 Consider  $M\in\mbc^{p\times p}$, $\mc{A}\in\mbc^{m\times k\times p}$, $\mc{B}\in\mbc^{s\times n\times p}$ and $\mc{C}\in\mbc^{m\times n\times p}$. Then the tensor equation \eqref{eq21} is consistent if and only if $\mc{A}\m\mc{A}^{(1)}\m\mc{C}\m\mc{B}^{(1)}\m\mc{B}=\mc{C}$. In this case, the general solution is given by 
\begin{equation}\label{eq22}
    \mc{X}=\mc{A}^{(1)}\m\mc{C}\m\mc{B}^{(1)}+\mc{Z}-\mc{A}^{(1)}\m \mc{A}\m\mc{Z}\m\mc{B}\m\mc{B}^{(1)},
\end{equation}
for some $\mc{Z}\in\mbc^{k\times s\times p}$.
\end{theorem}
\begin{proof}
 Let $\mc{A}\m\mc{A}^{(1)}\m\mc{C}\m\mc{B}^{(1)}\m\mc{B}=\mc{C}$ and $\mc{Z}=\mc{A}^{(1)}\m\mc{C}\m\mc{B}^{(1)}$. Then $\mc{Z}$ is a solution of equation \eqref{eq21} since 
 \[\mc{A}\m\mc{Z}\m\mc{B}=\mc{A}\m\mc{A}^{(1)}\m\mc{C}\m\mc{B}^{(1)}\m\mc{B}=\mc{C}.\]
 Conversely, if equation \eqref{eq21} has a solution say $\mc{Z}$, then 
 \[\mc{C}=\mc{A}\m\mc{Z}\m\mc{B}=\mc{A}\m\mc{A}^{(1)}\m\mc{A}\m\mc{Z}\m\mc{B}\m\mc{B}^{(1)}\m\mc{B}=\mc{A}\m\mc{A}^{(1)}\m\mc{C}\m\mc{B}^{(1)}\m\mc{B}.\]
 Clearly $\mc{A}^{(1)}\m\mc{C}\m\mc{B}^{(1)}+\mc{Z}-\mc{A}^{(1)}\m \mc{A}\m\mc{Z}\m\mc{B}\m\mc{B}^{(1)}$ is a solution equation \eqref{eq21}. If $\mc{Z}$ is any other solution of equation \eqref{eq21}, then $\mc{A}\m\mc{Z}\m\mc{B}=\mc{C}$ and 
 \begin{eqnarray*}
     \mc{Z}&=&\mc{A}^{(1)}\m\mc{C}\m\mc{B}^{(1)}+\mc{Z}-\mc{A}^{(1)}\m\mc{C}\m\mc{B}^{(1)}\\
     &=&\mc{A}^{(1)}\m\mc{C}\m\mc{B}^{(1)}+\mc{Z}-\mc{A}^{(1)}\m\mc{A}\m\mc{Z}\m\mc{B}\m\mc{B}^{(1)},
 \end{eqnarray*}
 and completes the proof.
\end{proof}
Now, we define the least-squares solution and minimum-norm least-squares solution \cite{jin2023} for the tensor equations $\mc{A}\m\mc{Z}\m\mc{B}=\mc{C}$. 
\begin{definition}
Consider  $M\in\mbc^{p\times p}$, $\mc{A}\in\mbc^{m\times k\times p}$, $\mc{B}\in\mbc^{s\times n\times p}$ and $\mc{C}\in\mbc^{m\times n\times p}$. A tensor $\mc{Z}_0$  is called a least-squares solution of  $\mc{A}\m\mc{Z}\m\mc{B}=\mc{C}$ if 
\[\|\mc{A}\m\mc{Z}_0\m\mc{B}-\mc{C}\|_M\leq \|\mc{A}\m\mc{Z}\m\mc{B}-\mc{C}\|_M,~\mbox{ for all }\mc{Z}\in\mbc^{k\times 1\times p}.\]
\end{definition}
\begin{definition}
Consider  $M\in\mbc^{p\times p}$, $\mc{A}\in\mbc^{m\times k\times p}$, $\mc{B}\in\mbc^{s\times n\times p}$ and $\mc{C}\in\mbc^{m\times n\times p}$. Define $\mc{S}_{lss}$ be the set of least-squares solutions of $\mc{A}\m\mc{Z}\m\mc{B}=\mc{C}$. A tensor $\mc{Z}_0$  is called the minimum-norm least-squares solution of  $\mc{A}\m\mc{Z}\m\mc{B}=\mc{C}$ if 
\[\|\mc{Z}_0\|_M\leq \|\mc{Z}\|_M,~\mbox{ for all }\mc{Z}\in \mc{S}_{lss}.\]
\end{definition}
\begin{theorem}\label{thm-4.3}
Consider  $M\in\mbc^{p\times p}$, $\mc{A}\in\mbc^{m\times k\times p}$, $\mc{B}\in\mbc^{s\times n\times p}$ and $\mc{C}\in\mbc^{m\times n\times p}$. Then $\mc{X}=\mc{A}^{\dagger}\m\mc{C}\m\mc{B}^{\dagger}$ is the minimum-norm least squares solution of the tensor equation \eqref{eq21}.
\end{theorem}
\begin{proof}
    Let $\mc{A}\m\mc{X}\m\mc{B}=\mc{C}$. Then $\tilde{\mc{C}}=\wtilde{(\mc{A}\m\mc{X}\m\mc{B})}=\tilde{\mc{A}}\Delta\tilde{\mc{X}}\Delta\tilde{\mc{B}}$ and consequently, 
\begin{equation}\label{eq23}
\tilde{\mc{C}}(:,:,i)=\tilde{\mc{A}}(:,:,i)\tilde{\mc{X}}(:,:,i)\tilde{\mc{B}}(:,:,i),~i=1,2,\ldots, p.    
\end{equation}
For each $i$, the  minimum-norm least squares solution of the matrix equation \eqref{eq23} is given by 
\begin{equation}\label{eq24}
\tilde{\mc{X}}(:,:,i)= \left[\tilde{\mc{A}}(:,:,i)\right]^{\dagger}\tilde{\mc{C}}(:,:,i)\left[\tilde{\mc{B}}(:,:,i)\right]^{\dagger}.
\end{equation}
By Proposition \ref{prop-tildag} (ii), equation \eqref{eq24} becomes
\[
\tilde{\mc{X}}(:,:,i)=(\tilde{\mc{A}})^{\dagger}(:,:,i)\tilde{\mc{C}}(:,:,i)(\tilde{\mc{B}})^{\dagger}(:,:,i) \iff \tilde{\mc{X}}=(\tilde{\mc{A}})^{\dagger}\Delta\tilde{\mc{C}}\Delta(\tilde{\mc{B}})^{\dagger}=\wtilde{\mc{A}^{\dagger}}\Delta\tilde{\mc{C}}\Delta\wtilde{\mc{B}^{\dagger}}.
\]
Equivalently, $\mc{X}=\tilde{\mc{X}}\times_3 M^{-1}=\left(\mc{A}^{\dagger}\m\mc{C}\m\mc{B}^{\dagger}\right)\times_3 M\times_3 M^{-1}=\mc{A}^{\dagger}\m\mc{C}\m\mc{B}^{\dagger}$ is the minimum-norm least squares of $\mc{A}\m\mc{X}\m\mc{B}=\mc{C}$.

\end{proof}
The following corollary can be obtained by applying Lemma \ref{tik-mat} and Theorem \ref{thm-4.3}.
\begin{corollary}\label{cor-4.5}
  Consider  $M\in\mbc^{p\times p}$, $\mc{A}\in\mbc^{m\times k\times p}$, $\mc{B}\in\mbc^{s\times n\times p}$ and $\mc{C}\in\mbc^{m\times n\times p}$. Let $\lambda,\mu$ ($>0$) be the regularization parameters. Then
  \begin{enumerate}
      \item[(i)] $\displaystyle\mc{X}_{\lambda\mu}=(\mc{A}^*\m\mc{A}+\lambda \mc{I}_{kkp})^{-1}\m\mc{A}^*\m\mc{C}\m\mc{B}^*\m(\mc{B}\m\mc{B}^*+\mu \mc{I}_{ssp})^{-1}\to\mc{A}^{\dagger}\m\mc{C}\m\mc{B}^{\dagger}$ as $\lambda\to 0$ and $\mu\to 0$.
      \item[(ii)] $ \displaystyle\mc{X}_{\lambda\mu}$ is the minimum-norm least squares solution of the tensor equation \eqref{eq21},  as $\lambda\to 0$ and $\mu\to 0$.
  \end{enumerate}
  
\end{corollary}

\begin{algorithm}[H]
\caption{Two-step Tikhonov's regularized solution for $\mc{A}\m\mc{X}\m\mc{B}=\mc{C}$} \label{alg:mtik}
\begin{algorithmic}[1]
\State {\bf Input} $M\in\mbc^{p\times p}$, $\mc{A}\in\mbc^{m\times k\times p}$, $\mc{B}\in\mbc^{s\times n\times p}$, $\mc{C}\in\mbc^{m\times n\times p}$, $\lambda$ and $\mu$
\State Compute $\mc{T}_{\lambda}=\left(\mc{A}^*\m\mc{A}+\lambda \mc{I}_{kkp}\right)^{-1},~\mc{T}_{\mu}=\left(\mc{B}\m\mc{B}^*+\lambda \mc{I}_{ssp}\right)^{-1}$,~$\mc{C}_1=\mc{A}^*\m\mc{C}\m\mc{B}^*$

\State Compute $\mc{X}_{\lambda\mu}=\mc{T}_{\lambda}\m\mc{C}_1\m\mc{T}_{\mu}$

\State \Return The regularized solution $\mc{X}_{\lambda\mu}$.
 \end{algorithmic}\label{Alg-reg}
\end{algorithm}

\subsection{Image deblurring}
The problem of color image reconstruction from a blurred observation can be formulated using a three-channel representation framework. Given a blurred color image $\mc{C} \in \mathbb{R}^{m \times n \times 3}$, we can decompose it into its constituent channels $\mc{C}^{(1)}, \mc{C}^{(2)}, \mc{C}^{(3)} \in \mathbb{R}^{m \times n}$, where each channel represents the intensity values for a specific color component. Our objective is to reconstruct the true, error-free color image $\mc{X} \in \mathbb{R}^{k \times l \times 3}$ by computing its corresponding channel components $X^{(1)}, X^{(2)}, X^{(3)} \in \mathbb{R}^{k \times l}$, which we have to compute using the following equation. 
\begin{equation}\label{debleq1}
\mc{A}\m\mc{X} \m \mc{B}= \mc{C}, ~~~\mc{A} \in \mathbb{R}^{m\times k \times 3},~\ \mc{X}\in \mathbb{R}^{k\times l \times 3}, ~\mc{B}\in \mathbb{R}^{l\times n \times 3} \text{~~and~~} \mc{C}\in \mathbb{R}^{m\times n \times 3} 
\end{equation}
The frontal slices of  $\mc{A}$ and $\mc{B}$ are considered as the  vertical and horizontal within-channel blurring matrices, respectively. In view of cross-channel blurring, the entries of $\mc{A}$ are defined by 
\begin{equation*}
\mc {A}(:,:,k)=\delta_k \mc{A}^{(k)}, \text{~for~} k=1,2,3, \text{~where~} \delta_k \in \mathbb{R} \text{~satisfying~} \displaystyle\sum_{k=1}^3\delta_k = 1,
\end{equation*}
and 
 \begin{equation*}
\mc{A}^{(k)}(i,j)  =\left\{\begin{array}{ll}
\frac{1}{\sigma_v \sqrt{2 \pi}} e^{-\frac{(i-j)^{2}}{2 \sigma_v^{2}}}, & |i-j| \leq b_v \\
0, & \text { otherwise }
\end{array}.\right.
 \end{equation*}
Here $\sigma_v$ controls the amount of smoothing, that is, the more ill-posed the problem is when $\sigma_v$ is larger and $b_v$ denotes the bandwidth.
Further, the entries of $\mc{B}$ are defined as $\mc{B}(:,:,1)=\mc{A}(:,:,1)^T$, and $\mc{B}(:,:,2)=\mc{B}(:,:,3)=0$.
\begin{figure}[H]
\begin{center}
\subfigure[]{\includegraphics[height=4.6cm]{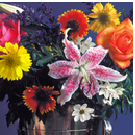}}~~
\subfigure[]{\includegraphics[height=4.6cm]{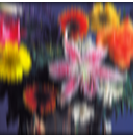}}~~
\subfigure[]{\includegraphics[height=4.6cm]{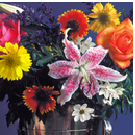}}
\caption{A visual comparison showcasing (a) the original true image at $128 \times 128$ resolution, (b) the degraded version affected by blurring and noise, and (c) the result of the reconstruction process aimed at restoring the image quality.
}\label{im1}
\subfigure[]{\includegraphics[height=4.6cm]{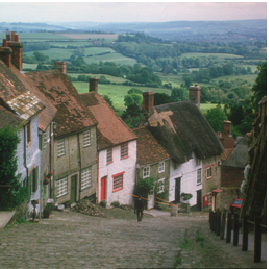}}~~
\subfigure[]{\includegraphics[height=4.6cm]{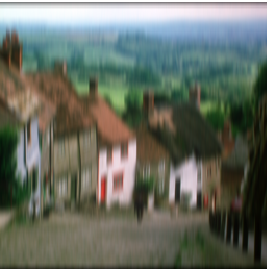}}~~
\subfigure[]{\includegraphics[height=4.6cm]{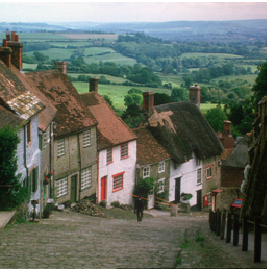}}
\caption{A visual comparison showcasing (a) the original true image at $256 \times 256$ resolution, (b) the degraded version affected by blurring and noise, and (c) the result of the reconstruction process aimed at restoring the image quality.
}\label{im2}
\end{center}
\end{figure}

In our numerical experiment, we focused on analyzing two color images with dimensions of $128\times 128\times 3 $ and $256\times 256\times 3 $ pixels. To simulate the effects of blurring, we used the following specific parameters: $\sigma_v=4$ and a blur factor $b_v=30$. The blurred image, denoted $\mc{B}$, was generated using the equation $\mc{B}=\mc{A}\m\mc{X}+\mc{N}$. Here, $\mc{N}$ represents a noise tensor created by applying Gaussian noise to each slice, with a mean of $0$ and a variance of $10^{-3}$. For the blurring process, we implemented cross-channel blurring, where the values were set as $\delta_1=0.75$, $\delta_2=\delta_3=0.25$. The original, unaltered images can be seen in Figure \ref{im1} (a) and Figure \ref{im2} (a), while the processed images, which have been blurred and affected by noise, are presented in Figure \ref{im1} (b) and Figure \ref{im2} (b).
To recover the true images from the altered versions, we applied a regularized approach as outlined in Algorithm \ref{Alg-reg}. The reconstructed image is shown in Figure \ref{im1} (c) and Figure \ref{im2} (c), which shows the efficiency of the algorithm. 

\section{Conclusion}\label{Section5}
We introduced a novel tensor-based two-step iterative method for $\mc{A}\m\mc{X}\m\mc{B}=\mc{C}$, featuring enhanced convergence through preconditioning and parametric optimization. This approach has proven effective in image deblurring applications and shows potential for solving Sylvester tensor equations. Future research directions include determining optimal parameter selections, analyzing the comparative performance of different splitting techniques under various conditions, and extending the parameterized iterative framework to address the Sylvester tensor equations.

\section*{Funding}
\begin{itemize}
\item Ratikanta Behera is supported by the Anusandhan National Research Foundation (ANRF), Government of India, under Grant No. EEQ/2022/001065.
\item Jajati Keshari Sahoo is supported by the Anusandhan National Research Foundation (ANRF), Government of India, under Grant No. SUR/2022/004357.
\end{itemize}

\section*{Conflict of Interest}
The authors assert with confidence that they have no conflicts of interest to disclose, ensuring the integrity of their work.
\section*{Data Availability}
In the context of this article, it is important to clarify that the data sets created or examined during the course of this study can be shared on request.


\bibliographystyle{abbrv}

\begin{thebibliography}{10}
\bibitem{barata2012moore}
J.~C.~A. Barata and M.~S. Hussein.
\newblock The moore--penrose pseudoinverse: A tutorial review of the theory.
\newblock {\em Brazilian Journal of Physics}, 42:146--165, 2012.

\bibitem{AsishRJ}
R.~Behera, A.~K. Nandi, and J.~K. Sahoo.
\newblock Further results on the {D}razin inverse of even-order tensors.
\newblock {\em Numer. Linear Algebra Appl.}, 27(5):e2317, 25, 2020.

\bibitem{behera2023computing}
R.~Behera, J.~K. Sahoo, P.~S. Stanimirovic, A.~Stupina, and A.~Stupin.
\newblock Computing tensor generalized bilateral inverses.
\newblock {\em Commun. Appl. Math. Comput.}, 2024, 10.1007/s40314-019-1007-1.

\bibitem{behera2023computation}
R.~Behera, J.~K. Sahoo, and Y.~Wei.
\newblock Computation of {O}uter {I}nverse of {T}ensors {B}ased on
  {$t$}-{P}roduct.
\newblock {\em Numer. Linear Algebra Appl.}, 32(1):Paper No. e2599, 2025.

\bibitem{berman}
A.~Berman and R.~J. Plemmons.
\newblock {\em Nonnegative matrices in the mathematical sciences}, volume~9 of
  {\em Classics in Applied Mathematics}.
\newblock Society for Industrial and Applied Mathematics (SIAM), Philadelphia,
  PA, 1994.
\newblock Revised reprint of the 1979 original.

\bibitem{BouJb08}
A.~Bouhamidi and K.~Jbilou.
\newblock A note on the numerical approximate solutions for generalized
  {S}ylvester matrix equations with applications.
\newblock {\em Appl. Math. Comput.}, 206(2):687--694, 2008.

\bibitem{Braman10}
K.~Braman.
\newblock Third-order tensors as linear operators on a space of matrices.
\newblock {\em Linear Algebra Appl.}, 433(7):1241--1253, 2010.

\bibitem{Brazell}
M.~Brazell, N.~Li, C.~Navasca, and C.~Tamon.
\newblock Solving multilinear systems via tensor inversion.
\newblock {\em SIAM J. Matrix Anal. Appl.}, 34(2):542--570, 2013.

\bibitem{Chen21}
F.~Chen and T.-Y. Li.
\newblock Two-step {AOR} iteration method for the linear matrix equation
  {$AXB=C$}.
\newblock {\em Comput. Appl. Math.}, 40(3):Paper No. 89, 12, 2021.

\bibitem{chen2023perturbations}
J.~Chen, W.~Ma, Y.~Miao, and Y.~Wei.
\newblock Perturbations of tensor-schur decomposition and its applications to
  multilinear control systems and facial recognitions.
\newblock {\em Neurocomputing}, 547:126359, 2023.

\bibitem{ChenMin20}
M.~Chen and D.~Kressner.
\newblock Recursive blocked algorithms for linear systems with {K}ronecker
  product structure.
\newblock {\em Numer. Algorithms}, 84(3):1199--1216, 2020.

\bibitem{ChenLu12}
Z.~Chen and L.~Lu.
\newblock A projection method and {K}ronecker product preconditioner for
  solving {S}ylvester tensor equations.
\newblock {\em Sci. China Math.}, 55(6):1281--1292, 2012.

\bibitem{perea1}
J.-J. Climent and C.~Perea.
\newblock Some comparison theorems for weak nonnegative splittings of bounded
  operators.
\newblock In {\em Proceedings of the {S}ixth {C}onference of the
  {I}nternational {L}inear {A}lgebra {S}ociety ({C}hemnitz, 1996)}, volume
  275/276, pages 77--106, 1998.

\bibitem{perea}
J.-J. Climent and C.~Perea.
\newblock Comparison theorems for weak nonnegative splittings of {$K$}-monotone
  matrices.
\newblock {\em Electron. J. Linear Algebra}, 5:24--38, 1999.

\bibitem{varga1}
G.~Csordas and R.~S. Varga.
\newblock Comparisons of regular splittings of matrices.
\newblock {\em Numer. Math.}, 44(1):23--35, 1984.

\bibitem{ding2005gradient}
F.~Ding and T.~Chen.
\newblock Gradient based iterative algorithms for solving a class of matrix
  equations.
\newblock {\em IEEE Transactions on Automatic Control}, 50(8):1216--1221, 2005.

\bibitem{DingTong05}
F.~Ding and T.~Chen.
\newblock Iterative least-squares solutions of coupled {S}ylvester matrix
  equations.
\newblock {\em Systems Control Lett.}, 54(2):95--107, 2005.

\bibitem{DuRuan22}
K.~Du, C.-C. Ruan, and X.-H. Sun.
\newblock On the convergence of a randomized block coordinate descent algorithm
  for a matrix least squares problem.
\newblock {\em Appl. Math. Lett.}, 124:Paper No. 107689, 8, 2022.

\bibitem{Einstein}
A.~Einstein.
\newblock The foundation of the general theory of relativity, in the collected
  papers of albert einstein 6, aj kox, mj klein, and r. schulmann, eds, 2007.

\bibitem{Golub79}
G.~H. Golub, S.~Nash, and C.~Van~Loan.
\newblock A {H}essenberg-{S}chur method for the problem {$AX+XB=C$}.
\newblock {\em IEEE Trans. Automat. Control}, 24(6):909--913, 1979.

\bibitem{Wensheng}
W.~Hu, Y.~Yang, W.~Zhang, and Y.~Xie.
\newblock Moving object detection using tensor-based low-rank and saliently
  fused-sparse decomposition.
\newblock {\em IEEE Trans. Image Process.}, 26(2):724--737, 2017.

\bibitem{Jidrazin}
J.~Ji and Y.~Wei.
\newblock The {D}razin inverse of an even-order tensor and its application to
  singular tensor equations.
\newblock {\em Comput. Math. Appl.}, 75(9):3402--3413.

\bibitem{jin2023}
H.~Jin, S.~Xu, Y.~Wang, and X.~Liu.
\newblock The {M}oore-{P}enrose inverse of tensors via the {M}-product.
\newblock {\em Comput. Appl. Math.}, 42(6):Paper No. 294, 28, 2023.

\bibitem{Kernfeldlinear}
E.~Kernfeld, M.~Kilmer, and S.~Aeron.
\newblock Tensor-tensor products with invertible linear transforms.
\newblock {\em Linear Algebra Appl.}, 485:545--570, 2015.

\bibitem{kilmer13}
M.~E. Kilmer, K.~Braman, N.~Hao, and R.~C. Hoover.
\newblock Third-order tensors as operators on matrices: a theoretical and
  computational framework with applications in imaging.
\newblock {\em SIAM J. Matrix Anal. Appl.}, 34(1):148--172, 2013.

\bibitem{kilm21}
M.~E. Kilmer, L.~Horesh, H.~Avron, and E.~Newman.
\newblock Tensor-tensor algebra for optimal representation and compression of
  multiway data.
\newblock {\em Proc. Natl. Acad. Sci. USA}, 118(28):Paper No. e2015851118, 12,
  2021.

\bibitem{kilmer11}
M.~E. Kilmer and C.~D. Martin.
\newblock Factorization strategies for third-order tensors.
\newblock {\em Linear Algebra Appl.}, 435(3):641--658, 2011.

\bibitem{Bader}
T.~G. Kolda and B.~W. Bader.
\newblock Tensor decompositions and applications.
\newblock {\em SIAM Rev.}, 51(3):455--500, 2009.

\bibitem{LiuGuo18}
M.~Liu, B.~Li, Q.~Guo, C.~Zhu, P.~Hu, and Y.~Shao.
\newblock Progressive iterative approximation for regularized least square
  bivariate {B}-spline surface fitting.
\newblock {\em J. Comput. Appl. Math.}, 327:175--187, 2018.

\bibitem{miwakeichi}
F.~Miwakeichi, E.~Mart{\i}nez-Montes, P.~A. Vald{\'e}s-Sosa, N.~Nishiyama,
  H.~Mizuhara, and Y.~Yamaguchi.
\newblock Decomposing eeg data into space--time--frequency components using
  parallel factor analysis.
\newblock {\em NeuroImage}, 22(3):1035--1045, 2004.

\bibitem{mqdr}
K.~Panigrahy, B.~Karmakar, J.~K. Sahoo, R.~Behera, and R.~N. Mohapatra.
\newblock Computation of {$M$-QDR} decomposition of tensors and applications,
  2024.

\bibitem{qi2017}
L.~Qi and Z.~Luo.
\newblock {\em Tensor analysis: spectral theory and special tensors}.
\newblock SIAM, 2017.

\bibitem{Ragnarsson}
S.~Ragnarsson and C.~F. Van~Loan.
\newblock Block tensor unfoldings.
\newblock {\em SIAM J. Matrix Anal. Appl.}, 33(1):149--169, 2012.

\bibitem{Regalia}
P.~A. Regalia and S.~K. Mitra.
\newblock Kronecker products, unitary matrices and signal processing
  applications.
\newblock {\em SIAM Rev.}, 31(4):586--613, 1989.

\bibitem{Sahoo}
J.~K. Sahoo, R.~Behera, P.~S. Stanimirovi\'{c}, V.~N. Katsikis, and H.~Ma.
\newblock Core and core-{EP} inverses of tensors.
\newblock {\em Comput. Appl. Math.}, 39(1):Paper No. 9, 28, 2020.

\bibitem{SahooPandaBehera25}
J.~K. Sahoo, S.~K. Panda, R.~Behera, and P.~S. Stanimirovi\'c.
\newblock Computation of tensors generalized inverses under {$M$}-product and
  applications.
\newblock {\em J. Math. Anal. Appl.}, 542(1):Paper No. 128864, 24, 2025.

\bibitem{Shao1}
J.-Y. Shao.
\newblock A general product of tensors with applications.
\newblock {\em Linear Algebra Appl.}, 439(8):2350--2366, 2013.

\bibitem{ShiWei13}
X.~Shi, Y.~Wei, and S.~Ling.
\newblock Backward error and perturbation bounds for high order {S}ylvester
  tensor equation.
\newblock {\em Linear Multilinear Algebra}, 61(10):1436--1446, 2013.

\bibitem{song}
Y.~Z. Song.
\newblock Comparisons of nonnegative splittings of matrices.
\newblock {\em Linear Algebra Appl.}, 154/156:433--455, 1991.

\bibitem{Stanimirovi}
P.~S. Stanimirovi\'{c}, M.~\'{C}iri\'{c}, V.~N. Katsikis, C.~Li, and H.~Ma.
\newblock Outer and {$(b,c)$} inverses of tensors.
\newblock {\em Linear Multilinear Algebra}, 68(5):940--971, 2020.

\bibitem{Baodong}
L.~Sun, B.~Zheng, C.~Bu, and Y.~Wei.
\newblock Moore-{P}enrose inverse of tensors via {E}instein product.
\newblock {\em Linear Multilinear Algebra}, 64(4):686--698, 2016.

\bibitem{TianLiu17}
Z.~Tian, M.~Tian, Z.~Liu, and T.~Xu.
\newblock The {J}acobi and {G}auss-{S}eidel-type iteration methods for the
  matrix equation {$AXB=C$}.
\newblock {\em Appl. Math. Comput.}, 292:63--75, 2017.

\bibitem{varga}
R.~S. Varga.
\newblock {\em Matrix iterative analysis}.
\newblock Prentice-Hall, Inc., Englewood Cliffs, NJ, 1962.

\bibitem{WangLi13}
X.~Wang, Y.~Li, and L.~Dai.
\newblock On {H}ermitian and skew-{H}ermitian splitting iteration methods for
  the linear matrix equation {$AXB=C$}.
\newblock {\em Comput. Math. Appl.}, 65(4):657--664, 2013.

\bibitem{wozi}
Z.~I. Wo\'znicki.
\newblock Nonnegative splitting theory.
\newblock {\em Japan J. Indust. Appl. Math.}, 11(2):289--342, 1994.

\bibitem{young}
D.~M. Young.
\newblock {\em Iterative solution of large linear systems}.
\newblock Dover Publications, Inc., Mineola, NY, 2003.
\newblock Unabridged republication of the 1971 edition [Academic Press, New
  York-London, MR 305568].

\bibitem{ZhouDuan08}
B.~Zhou and G.-R. Duan.
\newblock On the generalized {S}ylvester mapping and matrix equations.
\newblock {\em Systems Control Lett.}, 57(3):200--208, 2008.

\end{thebibliography}

\end{document}